\documentclass[12pt]{amsart}
\usepackage{ amsmath, amsthm, amsfonts, amssymb, color}
 \usepackage{mathrsfs}
\usepackage{amsfonts, amsmath}
 \usepackage{amsmath,amstext,amsthm,amssymb,amsxtra}
 \usepackage{txfonts} %also pxfonts
 \usepackage[colorlinks, citecolor=blue,pagebackref,hypertexnames=false]{hyperref}
 \allowdisplaybreaks
 \usepackage{pgf,tikz}
 \usepackage{multirow}%newtablepackage1
 \usepackage{diagbox} %newtablepackage2

 \usepackage{tikz}

\usetikzlibrary{decorations.pathreplacing}

\begin{document}

 \baselineskip 16.6pt
\hfuzz=6pt

\widowpenalty=10000

\newtheorem{cl}{Claim}
\newtheorem{theorem}{Theorem}[section]
\newtheorem{proposition}[theorem]{Proposition}
\newtheorem{coro}[theorem]{Corollary}
\newtheorem{lemma}[theorem]{Lemma}
\newtheorem{definition}[theorem]{Definition}
\newtheorem{assum}{Assumption}[section]
\newtheorem{example}[theorem]{Example}
\newtheorem{remark}[theorem]{Remark}
\renewcommand{\theequation}
{\thesection.\arabic{equation}}

\def\SL{\sqrt H}

\newcommand{\mar}[1]{{\marginpar{\sffamily{\scriptsize
        #1}}}}

\newcommand{\as}[1]{{\mar{AS:#1}}}

\newcommand\R{\mathbb{R}}
\newcommand\RR{\mathbb{R}}
\newcommand\CC{\mathbb{C}}
\newcommand\NN{\mathbb{N}}
\newcommand\ZZ{\mathbb{Z}}
\newcommand\HH{\mathbb{H}}
\newcommand\Z{\mathbb{Z}}
\def\RN {\mathbb{R}^n}
\renewcommand\Re{\operatorname{Re}}
\renewcommand\Im{\operatorname{Im}}

\newcommand{\mc}{\mathcal}
\newcommand\D{\mathcal{D}}
\def\hs{\hspace{0.33cm}}
\newcommand{\la}{\alpha}
\def \l {\alpha}
\newcommand{\eps}{\tau}
\newcommand{\pl}{\partial}
\newcommand{\supp}{{\rm supp}{\hspace{.05cm}}}
\newcommand{\x}{\times}
\newcommand{\lag}{\langle}
\newcommand{\rag}{\rangle}

\newcommand\wrt{\,{\rm d}}

\newcommand{\norm}[2]{|#1|_{#2}}
\newcommand{\Norm}[2]{\|#1\|_{#2}}

\title[]{Besov space, Schatten classes and commutators of Riesz transforms associated with the Neumann Laplacian}

\author{Zhijie Fan}
\address{Zhijie Fan, Department of Mathematics, Wuhan University, }
\email{ZhijieFan@whu.edu.cn}

\author{Michael Lacey}
\address{Michael Lacey, Department of Mathematics, Georgia Institute of Technology
Atlanta, GA 30332, USA}
\email{lacey@math.gatech.edu}

\author{Ji Li}
\address{Ji Li, Department of Mathematics, Macquarie University, Sydney}
\email{ji.li@mq.edu.au}

\author[M. N. Vempati]{Manasa N. Vempati}
\address{Manasa N. Vempati, Department of Mathematics, Georgia Institute of Technology
Atlanta, GA 30332, USA}
\email{nvempati6@gatech.edu}

\author{Brett D. Wick}
\address{Brett D. Wick, Department of Mathematics\\
         Washington University - St. Louis\\
         St. Louis, MO 63130-4899 USA
         }
\email{wick@math.wustl.edu}

  \date{\today}

 \subjclass[2010]{47B10, 42B20, 43A85}
\keywords{Schatten class, Riesz transform commutator, Neumann operator, Besov space, Nearly weakly orthogonal}

\begin{abstract}
This article provides a deeper study of the Riesz transform commutators associated with the Neumann Laplacian operator $\Delta_N$ on $\mathbb R^n$. Along the line of singular value estimates for Riesz transform commutators established by Janson--Wolff and Rochberg--Semmes, we establish a full range of Schatten-$p$ class characterization for these commutators.
\end{abstract}

\maketitle

\tableofcontents

\section{Introduction}
Originating in the works of Nehari \cite{Ne} and Calder\'on \cite{Cal}, the theory of Calder\'on--Zygmund operator commutators plays a crucial role in harmonic analysis, which connects closely to complex analysis, non-commutative analysis and operator theory, see for example \cite{CLMS,CRW,HLW,Hy,LMSZ}.
%the commutator $[f,T]$, defined by $[f,T]:=M_fT-TM_f$ for a singular integral operator $T$ and a multiplication operator $M_f$ with symbol $f$,
A milestone working on boundedness and compactness of Riesz transform commutators $[b,R_j]$ was due to Coifman, Rochberg and Weiss \cite{CRW} and to Uchiyama \cite{U}, respectively, where $R_j$ is the $j$-th Riesz transform on $\mathbb{R}^n$.
%. They showed that these commutators are bounded on $L^p(\mathbb{R}^n)$ if and only if $b\in BMO(\mathbb{R}^n)$. Furthermore,  showed that these commutators become compact operators if and only if $b$ belongs to $VMO(\mathbb{R}^n)$, the function space of vanishing mean oscillation, which is a subspace of $BMO(\mathbb{R}^n)$.
These two parts have been extensively studied in various settings with applications to compensated compactness \cite{CLMS}, two weight estimates \cite{HLW}, little Hankel in several complex variables \cite{FL}, Jacobian equations \cite{Hy} and so on.

As a deeper study of the previous work and motivated by the quantised derivatives in non-commutative geometry (introduced in \cite[IV]{Connes}, see also \cite{CST,LMSZ,MSX2018,MSX2019}), singular value estimates of Riesz transform commutators via Schatten class were investigated by many authors in different settings \cite{FLLarxiv,FLMSZ,JW,LMSZ,LLW,MSX2018,MSX2019,RS1988,RS}, which are of independent interest in harmonic analysis and connect strongly to non-commutative geometry.
The summary of this well-known result in the classical setting is as follows:

(1) In the case of dimension $n=1$ and the Hilbert transform, one has $[b,H]\in S^p $ if and only if $b\in {\rm B}_{p,p}^{1\over p}(\mathbb R)$, where $0<p<\infty$ (see \cite{P}).

(2) In the case of dimension $n\geq 2$,  one has $[b,R_j]\in S^p $ if and only if $b\in {\rm B}_{p,p}^{n\over p}(\mathbb R^n)$ when $p>n$, whereas $[b,R_j]\in S^p $ if and only if $b$ is a constant when $0<p\leq n$ (see \cite{JW,RS1988}).

Here ${\rm B}_{p,p}^{n\over p}(\mathbb R^n)$ is the homogeneous Besov space in $\mathbb R^n$, $n\geq1$.  $S^p$, $0<p<\infty$, is the Schatten-$p$ class, defined as follows:  let $T$ be any compact operator on $L^{2}(\mathbb{R}^{n})$, then $T\in S^{p}$, if $\{\lambda_{n}(T)\}\in \ell^{p}$, where $\lambda_{n}(T)$ is the sequence of square roots of eigenvalues of $T^{*}T$. More equivalent characterizations and properties about Schatten classes can be found in e.g. \cite{LSZ1,LSZ2}.

 %Then a natural question occur: could one establish boundedness, compactness and Schatten-$p$ class characterization for Riesz transform commutator associated with an operator $L$ instead of the Laplacian? Note that \cite[Theorem 1.4]{DHLWY} indicated that if $L$ is the Dirichlet Laplacian $\Delta_{D_+}$ on $\mathbb{R}^n_+$, then the $BMO_{\Delta_{D_+}}(\mathbb{R}^n_+)$ cannot be characterized by the $L^p$ boundedness of $[b,R_{D_+,j}]^{-1/2}$, where $R_{D_+,j}=\frac{\partial}{\partial x_j}\Delta_{D_+}^{-1/2}$.

Let $\Delta_N$ be Neumann Laplacian operator on $\mathbb R^n$, $R_{N,j}={\partial\over\partial x_j} \Delta_N^{-1/2}$, $j=1,2,\ldots,n$ be the Riesz transforms  associated to $\Delta_N$  (defined in Section \ref{Neu}), and
\begin{align*}
[b,R_{N,j}](f)(x):= b(x)R_{N,j}(f)(x) - R_{N,j}(bf)(x).
\end{align*}
The boundedness and compactness characterization of $[b,R_{N,j}]$ were established in Li--Wick \cite{LW} (see also \cite{DGKLWY}) and Cao--Yabuta \cite{CY} , respectively. Thus, along the line of \cite{JW,RS1988},  a natural question occurs: \  ``could one establish the Schatten-$p$ class characterization for $[b,R_{N,j}]$?''

To study this, we consider the following sub-questions:
%Recently, Riesz transform commutators associated with a certain operator has attracted the interest of many researchers (see for example \cite{CY,CDLW,DGKLWY,DLLW,DLMWY,DLWY,LW}).
%In particular, motivated by the work of Coifman--Rochberg--Weiss and Uchiyama, Li--Wick \cite{LW} (see also \cite{DGKLWY}) and Cao--Yabuta \cite{CY} established the boundedness and compactness characterization for Riesz transform commutators associated with Neumann Laplacian $\Delta_N$, respectively. Then a natural question occurs: could one establish the Schatten-$p$ class characterization for $[b,R_{N,j}]$? This question can be divided into three sub-questions.

\noindent {\bf Sub-question 1:} Which Besov space is suitable to characterize the $S^p$ norm, $p>n$, of $[b,R_{N,j}]$?

\noindent {\bf Sub-question 2:} What is the relationship between this Besov space and the classical one?

\noindent {\bf Sub-question 3:} Do functions in this type of Besov space also collapse to constants when $p\leq n$?

In the last few decades, the theory of Besov spaces has been an active area of research, which, in particular, is useful to characterize the Schatten-$p$ class property of Riesz transform commutators (see for example \cite{FLLarxiv,JW,RS}). Among these works, it would be worthwhile to mention that Bui, Duong and Yan \cite{BDY} laid the foundation of the theory of Besov space associated with a certain operator $L$ under the assumption that $L$ generate an analytic semigroup $e^{-tL}$ with Gaussian upper bound on $L^2(X)$, where $X$ is a quasi-metric space of polynomial upper bounds on volume growth (see also \cite{BBD,MR3319540,CG,Hu,LYY} for other development along this direction). Inspired by their work, we will use the Besov space associated with Neumann operator as a suitable substitution of classical Besov space. To be more  precise, we set
%
%For $1\leq p,q< \infty$, $0<\alpha<1$ and any function $f\in L_{{\rm loc}}^1(\mathbb{R}_{\pm}^n)$, we define
%$$\|f\|_{B_{p,q}^{\alpha,e}(\mathbb{R}_{\pm}^n)}:=\|f_e\|_{B_{p,q}^{\alpha}(\mathbb{R}^n)}\ \ {\rm and}\ \ \|f\|_{B_{p,q}^{\alpha,o}(\mathbb{R}_{\pm}^n)}:=\|f_o\|_{B_{p,q}^{\alpha}(\mathbb{R}^n)}.$$
$$\mathcal{M}(\mathbb{R}^n):=\left\{f\in L_{{\rm loc}}^1(\mathbb{R}^n):\exists \epsilon>0\ {\rm s.t.}\int_{\mathbb{R}^n}\frac{|f(x)|^2}{1+|x|^{n+\epsilon}}<+\infty\right\}.$$
\begin{definition}
Suppose $1\leq p,q< \infty$ and $0<\alpha<1$.
$B_{p,q}^{\alpha,\Delta_{N}}(\mathbb{R}^{n}):=\{f\in \mathcal{M}(\mathbb{R}^n):\|f\|_{B_{p,q}^{\alpha,\Delta_{N}}(\mathbb{R}^{n})}<\infty \}$, where
\begin{align*}
\|f\|_{B_{p,q}^{\alpha,\Delta_{N}}(\mathbb{R}^{n})}:=\left(\int_0^\infty (t^{-\alpha}\|t\Delta_{N}e^{-t\Delta_N}f\|_{L^p(\mathbb{R}^n)})^q\frac{dt}{t}\right)^{1/q}.
\end{align*}
\end{definition}

%We will provide an equivalent characterization of $B_{p,q}^{\alpha,\Delta_N}(\mathbb{R}^{n})$ in Section \ref{chara}, which can naturally establish the relationship between Besov space $B_{p,q}^{\alpha,\Delta_N}(\mathbb{R}^{n})$ and Neumann operator $\Delta_N$. More backgrounds about Besov space can be found in Section \ref{chara}.

Now we provide our main result as follows.
\begin{theorem}\label{schatten}
Suppose  $n\geq2$,  $0<p<\infty$ and $b\in \mathcal{M}(\mathbb{R}^n)$. Then for any $\ell\in \{1,2,\ldots,n\}$, one has  $[b,R_{N,l}]\in S^p$
if and only if

%%  ENUMERATE
\begin{enumerate}
\item $b\in B_{p,p}^{\frac{n}{p},\Delta_N}(\mathbb R^n)$ when $p>n$; in this case we have $\|b\|_{B_{p,p}^{\frac{n}{p},\Delta_N}(\mathbb R^n)}\approx \|[b,R_{N,\ell}]\|_{S^p}$;

\item $b$ is a constant $c_1$ on $\mathbb{R}^n_+$ and another constant $c_2$ on $\mathbb{R}^n_-$ (in the sense of almost everywhere) when $0<p\leq n$, where $c_1$ and $c_2$ may not be the same.
\end{enumerate}
%% ENUMERATE

\end{theorem}
Comparing to the classical setting of Riesz transform commutator, the two main difficulties occur in the Neumann Laplacian setting:

$\bullet$ First, the Riesz transform kernel under consideration is of non-convolution type, so one cannot apply Fourier analysis as in the classical setting. To overcome this, we will adapt a new idea developed recently by the first three authors in \cite{FLLarxiv} to provide a refined lower bound of Riesz transform kernel, to apply the median of the symbol on the atoms of the martingale, the bootstrapping techniques, and the effective tool of {\it nearly weakly orthogonal} due to Rochberg--Semmes \cite{RS} to estimate the Schatten-$p$ norm.

$\bullet$ The second difficulty is a technical one: although the whole underlying space is $\mathbb{R}^n$, the Riesz transform kernel associated with Neumann Laplacian operator becomes a Calder\'on--Zygmund operator satisfying certain non-degenerate conditions (a suitable lower bound) only on $(\mathbb{R}_{+}^n\times\mathbb{R}_{+}^n)\cup (\mathbb{R}_{-}^n\times\mathbb{R}_{-}^n)$. As a consequence, the translation of a
system of dyadic cubes along the $x_n$ direction may go out of this range and then can no longer be a new system of dyadic cubes over this set. To overcome this, we will regard the half-plane as a space of homogeneous type and then apply
collection of adjacent systems of dyadic cubes developed in \cite{HK} to develop a new idea.

\smallskip

The paper is organized as follows. Section \ref{sec2} consists of three parts: the first part recalls the concept of adjacent systems of dyadic cubes and Haar basis on  spaces of homogenous type; the second part provides the definition of Neumann Laplacian operator and a refined lower bound of its associated Riesz transform kernel; the third part establishes several fundamental properties of the Besov space $B_{p,q}^{\alpha,\Delta_N}(\mathbb{R}^{n})$, including a useful equivalent characterization of this Besov space, an embedding theorem and an interpolation theorem. In Sections \ref{three} and \ref{four}, we give   the proof of Theorem \ref{schatten} for the cases $p>n$ and $0<p\leq n$, respectively, which lies in Propositions \ref{schattenlarge1}, \ref{schattenlarge2} and \ref{kkkey}.

Throughout the paper we denote  by $\chi_{E}$ the
indicator function of a subset $E\subseteq X$. We use $A\lesssim B$ to denote the statement that $A\leq CB$ for some constant $C>0$, and $A\simeq B$ to denote the statement that $A\lesssim B$ and $B\lesssim A$. For simplicity, we will usually abuse the notation $\pm$ to denote $+$ or $-$.

\section{Preliminaries}\label{sec2}
\setcounter{equation}{0}

\subsection{Preliminaries on Spaces of Homogeneous Type}\label{s2}
\noindent

In the proof of necessity (the lower bound) for the case $p>n$, we will regard the half-plane as a space of homogeneous type, in the sense of Coifman and Weiss (\cite{CWbook}), with Euclidean metric and Lebesgue measure. Specifically, for any $x\in\mathbb{R}_{\pm}^n$ and $r>0$, the set $B_{\mathbb{R}_{\pm}^{n}}(x,r):=B(x,r)\cap \mathbb{R}_{\pm}^n $, where $B(x,r)$ is a Euclidean ball with centre $x$ and radius $r$, is considered as a ball in $\mathbb{R}_{\pm}^{n}$, which satisfies the doubling condition stated as follow:  for all $x\in\mathbb{R}_{\pm}^{n}$ and $r>0$,
$$|B_{\mathbb{R}_{\pm}^n}(x,2r)|\leq 2^{n+1}|B_{\mathbb{R}_{\pm}^{n}}(x,r)|<\infty.$$

In what follows, for the convenience of the readers, we collect some properties about systems of dyadic cubes on homogeneous space and adapt it to the half-plane $\mathbb{R}_{\pm}^n$.
%\subsection{A System of Dyadic Cubes}\label{sec:dyadic_cubes}
A
countable family
$
    \mathcal{D}_{\pm}
    := \cup_{k\in\mathbb{Z}}\mathcal{D}_{k,\pm}, \
    \mathcal{D}_{k,\pm}
    :=\{Q^k_{\alpha,\pm}\colon \alpha\in \mathcal{A}_k\},
$
of Borel sets $Q^k_{\alpha,\pm}\subseteq \mathbb{R}_{\pm}^n$ is called \textit{a
system of dyadic cubes} over $\mathbb{R}^n_\pm$  with parameter $\delta\in (0,1)$  if it has the following properties:

\smallskip
\smallskip

(I) $    \mathbb{R}_{\pm}^n
    = \bigcup_{\alpha\in \mathcal{A}_k} Q^k_{\alpha,\pm}
    \quad\text{(disjoint union) for all}~k\in\Z$;

\smallskip
\smallskip

(II) $\text{If }\ell\geq k\text{, then either }
        Q^{\ell}_{\beta,\pm}\subseteq Q^k_{\alpha,\pm}\text{ or }
        Q^k_{\alpha,\pm}\cap Q^{\ell}_{\beta,\pm}=\emptyset$;

\smallskip
\smallskip

(III) $
    \text{For each }(k,\alpha)\text{ and each } \ell\leq k,
    \text{ there exists a unique } \beta
    \text{ such that }Q^{k}_{\alpha,\pm}\subseteq Q^\ell_{\beta,\pm};
$

\smallskip
\smallskip

(IV) For each $(k,\alpha)$ there exists at most $M$
        (a fixed geometric constant)  $\beta$ such that
    $$ Q^{k+1}_{\beta,\pm}\subseteq Q^k_{\alpha,\pm},\ {\rm and}\
        Q^k_{\alpha,\pm} =\bigcup_{{\substack{Q\in\mathcal{D}_{k+1,\pm}\\ Q\subseteq Q^k_{\alpha,\pm}}}}Q;$$

\smallskip
\smallskip

(V) For each $(k,\alpha)$, one has $$B_{\mathbb{R}_{\pm}^n}(x^k_{\alpha,\pm},\frac{1}{12}\delta^k)
    \subseteq Q^k_{\alpha,\pm}\subseteq B_{\mathbb{R}_{\pm}^n}(x^k_{\alpha,\pm},4\delta^k)
    =: B_{\mathbb{R}_{\pm}^n}(Q^k_{\alpha,\pm});$$

\smallskip
\smallskip

(VI) If $\ell\geq k$ and
   $Q^{\ell}_{\beta,\pm}\subseteq Q^k_{\alpha,\pm}$, then
   $$B_{\mathbb{R}_{\pm}^n}(Q^{\ell}_{\beta,\pm})\subseteq B_{\mathbb{R}_{\pm}^n}(Q^k_{\alpha,\pm}).$$

\smallskip
\smallskip

The set $Q^k_{\alpha,\pm}$ is called a \textit{dyadic cube of
generation} $k$ with centre point $x^k_{\alpha,\pm}\in Q^k_{\alpha,\pm}$
and sidelength~$\delta^k$. The family $\mathcal{D}:=\mathcal{D}_+\cup\mathcal{D}_-$ is called \textit{a
system of dyadic cubes} over $\mathbb{R}^n$ with parameter $\delta\in (0,1)$ .

%The interior and closure of
%$Q^k_\alpha$ are denoted by $\widetilde{Q}^k_{\alpha}$ and
%$\bar{Q}^k_{\alpha}$, respectively.

From the properties of the dyadic system, one can deduce that there exists a constant
$C_{0}>0$, such that for any $Q^k_{\alpha,\pm}$ and $Q^{k+1}_{\beta,\pm}$  with $Q^{k+1}_{\beta,\pm}\subset Q^k_{\alpha,\pm}$,
\begin{align}\label{Cmu0}
|Q^{k+1}_{\beta,\pm}|\leq |Q^k_{\alpha,\pm}|\leq C_0| Q^{k+1}_{\beta,\pm}|.
\end{align}

In particular, one may construct a
system of dyadic cubes on $\mathbb{R}_{\pm}^n$ in a standard way. To illustrate this, we let $\mathcal{D}_{\pm}^0:=\cup_{k\in\mathbb{Z}}\mathcal{D}_{k,\pm}^0$, where $\mathcal{D}_{k,\pm}^0$ is the standard dyadic partition of $\mathbb{R}^n_{\pm}$ into cubes with vertices at $\{(2^{-k}m_1,\ldots,2^{-k}m_n):(m_1,\ldots,m_n)\in\mathbb{Z}^{n-1}\times (\pm\mathbb{N})\}$. Then $\mathcal{D}^0:=\mathcal{D}_{+}^0\cup \mathcal{D}_{-}^0$ is a standard
system of dyadic cubes on $\mathbb{R}^n$. For any $k\in\mathbb{Z}$, write $\mathcal{D}_k^0=\mathcal{D}_{k,+}^0\cup\mathcal{D}_{k,-}^0$.
%\subsection{Adjacent Systems of Dyadic Cubes}\label{s3ss}

A
finite collection $\{\mathcal{D}_{\pm}^\nu\colon \nu=1,2,\ldots ,\kappa\}$ of the dyadic
families  is called a collection of
adjacent systems of dyadic cubes over $\mathbb{R}_{\pm}^{n}$ with parameters $\delta\in
(0,1) $ and $1\leq C_{adj}<\infty$ if it has the
following properties: individually, each $\mathcal{D}_{\pm}^\nu:=\cup_{k\in\mathbb{Z}}\mathcal{D}_{k,\pm}^\nu$ is a
system of dyadic cubes with parameter $\delta\in (0,1)$; collectively, for each ball
$B_{\mathbb{R}_{\pm}^n}(x,r)\subseteq \mathbb{R}_{\pm}^{n}$ with $\delta^{k+3}<r\leq\delta^{k+2},
k\in\Z$, there exist $\nu \in \{1, 2, \ldots, \kappa\}$ and
$Q\in\mathcal{D}_{k,\pm}^\nu$ of generation $k$ and with centre point
$^\nu x^k_{\alpha,\pm}$ such that $|x-{}^\nu x_{\alpha,\pm}^k| <
2\delta^{k}$ and
\begin{equation}\label{eq:ball;included}
    B_{\mathbb{R}_{\pm}^n}(x,r)\subseteq Q\subseteq B_{\mathbb{R}_{\pm}^n}(x,C_{adj}r).
\end{equation}

We recall from \cite{HK} the following construction.

\begin{lemma}\label{thm:existence2}
On $\mathbb{R}_{\pm}^{n}$ with Euclidean metric and Lebesgue measure, there exists a collection $\{\mathcal{D}_{\pm}^\nu\colon
    \nu = 1,2,\ldots ,\kappa\}$ of adjacent systems of dyadic cubes with
    parameters $\delta\in (0, \frac{1}{96}) $ and ${C_{adj}} := 8\delta^{-3}$. The centre points
    $^\nu x^k_{\alpha,\pm}$ of the cubes $Q\in\mathcal{D}^\nu_{k,\pm}$ have, for each
    $\nu \in\{1,2,\ldots,\kappa\}$, the two properties
    \begin{equation*}
        |^\nu x_{\alpha,\pm}^k- {}^\nu x_{\beta,\pm}^k|
        \geq \frac{1}{4}\delta^k\quad(\alpha\neq\beta),\qquad
        \min_{\alpha}|x-{}^\nu x^k_{\alpha,\pm}|
        < 2 \delta^k\quad \text{for all}~x\in \mathbb{R}_{\pm}^n.
    \end{equation*}
    Moreover, these adjacent systems can be constructed in such a
    way that each $\mathcal{D}_{\pm}^\nu$ satisfies the distinguished
    centre point property: given a fixed point $x_{0,\pm}\in \mathbb{R}_{\pm}^n$, for every $k\in \Z$, there exists $\alpha\in\mathcal{A}_k$ such that
    $x_{0,\pm}
        = x^k_{\alpha,\pm},\text{ the centre point of }
        Q^k_{\alpha,\pm}\in\mathcal{D}_{k,\pm}^\nu.$
\end{lemma}
We will use the notion of \emph{nearly weakly orthogonal  (NWO)} sequences of functions proposed by Rochberg and Semmes \cite{RS}. For our purposes, we do not need to recall the explicit definition
of NWO sequences. Instead, it suffices to recall the below inequality: given a
system of dyadic cubes $\mathcal{D}$ over $\mathbb{R}^n$,  then for any
bounded compact operator $ T $ on $ L ^2 (\mathbb R ^{n})$:
\begin{equation}\label{e:NWO}
\Bigl[
\sum_{Q\in \mathcal{D}} |\langle T e_Q, f_Q \rangle | ^{p}
\Bigr] ^{1/p} \lesssim \| T \| _{S ^{p}},
\end{equation}
where $\{e_Q\}_{Q\in\mathcal{D}}$ and $\{f_Q\}_{Q\in\mathcal{D}}$ are function sequences satisfying $ \lvert  e_Q\rvert,  \lvert  f_Q\rvert  \leq \lvert  Q\rvert ^{-1/2} \chi_{cQ} $ for some $c>0$.
 %and $\mathcal{D}^\nu$ could be chosen as the union of $\mathcal{D}_{+}^\nu$ and $\mathcal{D}_{-}^\nu$.
This property can be found in  \cite[(1.10), \S3]{RS}.

%
%We recall from \cite[Remark 2.8]{KLPW}
%that the number $\kappa$ of the adjacent systems of dyadic
%    cubes as in the theorem above satisfies the estimate
%    \begin{equation*}%\label{eq:upperbound}
%        \kappa
%        = \kappa(A_0,\widetilde A_1,\delta)
%        \leq \widetilde A_1^6(A_0^4/\delta)^{\log_2\widetilde A_1},
%    \end{equation*}
%where $\widetilde A_1$ is the geometrically doubling constant, see \cite[Section 2]{KLPW}.

%Further, we also recall the following result on the smallness of the
%boundary.
%\begin{prop}
%    Suppose that $144A_0^8\delta\leq 1$. Let $\mu$ be a
%    positive $\sigma$-finite measure on $X$. Then the
%    collection $\{\mathscr{D}^t\colon t=1,2,\ldots ,T\}$ may be
%    chosen to have the additional property that
%    \[
%        \mu(\partial Q) = 0
%        \quad  \textup{ for all } \; Q\in\bigcup_{t=1}^{T}\mathscr{D}^t.
%    \]
%\end{prop}
%

%\subsection{An Explicit Haar Basis}\label{haardef}
For any $h\in B(0,1)$, we note that  the $h$-translated family $\tau^h\mathcal{D}:=\tau^h\mathcal{D}_+\cup\tau^h\mathcal{D}_-$ is a
system of dyadic cubes over $\mathbb{R}^n$ with parameter $\delta\in (0,1)$. We recall  the explicit construction in \cite{KLPW} of a Haar basis associated to the dyadic cubes
$Q\in\tau^h\mathcal{D}_k:=\tau^h\mathcal{D}_{k,+}\cup\tau^h\mathcal{D}_{k,-}$ as follows. Denote $M_Q := \#\mathcal H(Q) = \# \{R\in
\tau^h\mathcal{D}_{k+1,\pm}\colon R\subseteq Q\}$ be the number of
dyadic sub-cubes (``children''); namely $\mathcal{H}(Q)$ is the collection of dyadic children of $Q$. Then for any $Q\in \tau^h\mathcal{D}_{k}$, we let $h_{Q}^{1}$, $h_{Q}^{2},\ldots, h_{Q}^{M_Q-1}$ be a family of Haar functions which satisfy the properties collected in the following two lemmas.

\begin{lemma}[\cite{KLPW}]\label{thm:convergence}
For and $h\in B(0,1)$ and each $f\in
    L^p(\mathbb{R}^n)$, we have
    \[
        f(x)
        =  \sum_{Q\in\tau^h\mathcal{D}}\sum_{\epsilon=1}^{M_Q-1}
            \langle f,h^{\epsilon}_{Q}\rangle h^{\epsilon}_{Q}(x), %\quad \text{if } \mu(X)=\infty,
    \]
    where the sum converges (unconditionally) both in the
    $L^p(\mathbb{R}^n)$-norm and pointwise almost everywhere.
 \end{lemma}

\begin{lemma}[\cite{KLPW}]\label{prop:HaarFuncProp}
For any $h\in B(0,1)$, the Haar functions $h_{Q}^{\epsilon}$, where $Q\in\tau^h\mathcal{D}$
    and $\epsilon\in\{1,2,\ldots,M_Q - 1\}$, have the following properties:
    \begin{itemize}
        \item[(i)] $h_{Q}^{\epsilon}$ is a simple Borel-measurable
            real function on $\mathbb{R}^n$;
        \item[(ii)] $h_{Q}^{\epsilon}$ is supported on $Q$;
        \item[(iii)] $h_{Q}^{\epsilon}$ is constant on each
            $R\in\mathcal{H}(Q)$;
        \item[(iv)] $\int_{\mathbb{R}^n} h_{Q}^{\epsilon}\, dx= 0$ (cancellation);
        \item[(v)] $\langle h_{Q}^{\epsilon},h_{Q}^{\epsilon'}\rangle = 0$ for
            $\epsilon \neq \epsilon'$, $\epsilon$, $\epsilon'\in\{1, \ldots, M_Q - 1\}$;
        \item[(vi)] The collection
            $
                \big\{|Q|^{-1/2}\chi_Q\big\}
                \cup \{h_{Q}^{\epsilon} : \epsilon = 1, \ldots, M_Q - 1\}
            $
            is an orthogonal basis for the vector
            space~$V(Q)$ of all functions on $Q$ that are
            constant on each sub-cube $R\in\mathcal{H}(Q)$;
        \item[(vii)] %for $u = 1$, \ldots, $M_Q - 1$,
        If $h_{Q}^{\epsilon}\not\equiv 0$ then
            $
                \Norm{h_{Q}^{\epsilon}}{L^p(\mathbb{R}^n)}
                \approx |Q |^{\frac{1}{p} - \frac{1}{2}}
                \quad \text{for}~1 \leq p \leq \infty;
            $
        \item[(viii)] %for $u = 1$, \ldots, $M_Q - 1$,
                \hspace{4cm}
                $\Norm{h_{Q}^{\epsilon}}{L^1(\mathbb{R}^n)}\cdot
                \Norm{h_{Q}^{\epsilon}}{L^\infty(\mathbb{R}^n)} \approx 1$.
    \end{itemize}
\end{lemma}
%As stated in \cite{KLPW}, we also have $h_Q^0:= |Q|^{-1/2}1_Q$ which is a non-cancellative Haar function.
%Moreover, the martingale associated with the Haar functions are as follows: for $Q \in \mathcal{D}_k$,
%$$ \mathbb{E}_Qf = \langle f,h_Q^0\rangle h_Q^0
%%\quad \mathbb{D}_Qf = \sum_{\epsilon=1}^{M_Q-1} \langle f,h_{Q}^{\epsilon}\rangle h_{Q}^{\epsilon}
%\quad {\rm and}\quad
%\mathbb{D}_Qf =\sum_{\epsilon=1}^{M_Q-1} \mathbb{D}_{Q}^{\epsilon}f, $$
%where $\mathbb{D}_{Q}^{\epsilon}=\langle f,h_{Q}^{\epsilon}\rangle h_{Q}^{\epsilon}$ is the martingale operator associated with the $\epsilon$-th subcube of $Q$. Also we have
%$$
%   \mathbb{E}_kf =\sum_{Q\in \mathcal{D}_k}\mathbb{E}_Qf \quad{\rm and} \quad \mathbb{D}_kf = \mathbb{E}_{k+1}f- \mathbb{E}_kf.
%$$
%Hence, based on the construction of Haar system $\{h_Q^{\epsilon}\}$ in \cite{KLPW} we obtain that for each $R\in\mathcal D$,
%\begin{align*}%\label{e1}
%\sum_{Q:\ R\subset Q} \sum_{\epsilon=1}^{M_{Q}-1} \langle f, h_Q^{\epsilon}\rangle h_Q^{\epsilon} h_R^\eta= \sum_{Q:\ R\subset Q} \mathbb{D}_Qf\cdot h_R^{\eta}
%= \mathbb{E}_Rf\cdot h_R^{\eta} = \langle f,h_R^0\rangle h_R^0h_R^{\eta}.
%\end{align*}

\subsection{The Neumann Laplacian and its Associated Riesz Transform Kernel}\label{Neu}
Recall from \cite[(7), page 59 in Section 3.1]{S} that the Neumann problem on the half line $(0,\infty)$ is formulated as follows:
\begin{align}\label{Neumann}
\left\{
\begin{array}{lcc}
 u_t-u_{xx}  =0, & {\rm for\ }  0<x<\infty, 0<t<\infty, \\
 u(x,0)=f(x), &  \\
 u_x(0,t)=0. &
\end{array}
\right.
\end{align}
Then according to \cite[(7), Section 3.1]{S}, the solution can be expressed as
$$ u(x,t) = e^{-t\Delta_{1,N_+}}f(x),$$
where we denote the corresponding Laplacian in the Neumann problem \eqref{Neumann} by $\Delta_{1,N_+}$.

 For $n\geq 2$, we write $\mathbb{R}_+^n= \mathbb{R}^{n-1}\times \mathbb{R}_+$ and then follow the notations in \cite{CY,DDSY,DGKLWY,LW} to define the Neumann Laplacian on $\mathbb{R}^n_+$ by
$$ \Delta_{n,N_+} = \Delta_{n-1} + \Delta_{1,N_+}, $$
where $\Delta_{n-1}$ is the  Laplacian on $\mathbb{R}^{n-1}$.  Similarly we define the Neumann Laplacian  $\Delta_{n,N_-}$ on $\mathbb{R}^n_-$. For simplicity, in the remainder of this article, we will skip the lower index $n$ appeared in $\Delta_{n,N_-}$. We denote by $\Delta$ the Laplacian on $\mathbb{R}^n$ and denote the Neumann Laplacian on $\mathbb{R}^n_+$ (resp. $\mathbb{R}^n_-$) by $\Delta_{N_+}$ (resp. $\Delta_{N_-}$). Next, let $\Delta_N$ be the uniquely determined unbounded operator acting on $L^2(\mathbb{R}^n)$ such that
\begin{align}\label{Delta N}
 (\Delta_Nf)_+=\Delta_{N_+}f_+ \ \ \ {\rm and}\ \ \ (\Delta_Nf)_-=\Delta_{N_-}f_-
\end{align}
for all function $f: \mathbb{R}^n\rightarrow \mathbb{R}$ such that $f_+:=f|_{\mathbb{R}_{+}^n}\in W^{1,2}(\mathbb{R}^n_+)$ and $f_-:=f|_{\mathbb{R}_{-}^n}\in W^{1,2}(\mathbb{R}^n_-)$.

Observe that $\Delta$, $\Delta_{N_\pm}$ and $\Delta_N$ are positive self-adjoint operators. By the spectral theorem, one can define the heat semigroups $\{e^{-t\Delta}\}_{t\geq 0}$,  $\{e^{-t\Delta_{N_\pm}}\}_{t\geq 0}$ and $\{e^{-t\Delta_{N}}\}_{t\geq 0}$. Denote by $p_t(x,y)$, $p_{t,\Delta_{N_\pm}}(x,y)$ and $p_{t,\Delta_{N}}(x,y)$ the heat kernels corresponding to the heat semigroups generated by $\Delta$, $\Delta_{N_\pm}$ and $\Delta_N$, respectively. Then we have
$$p_t(x,y)=\frac{1}{(4\pi t)^{\frac{n}{2}}}e^{-\frac{|x-y|^2}{4t}}.$$
By the reflection method \cite[(7), (9), page 60 in Section 3.1]{S}, one can get
\begin{align*}
p_{t,\Delta_{N_{\pm}}}(x,y)=\frac{1}{(4\pi t)^{\frac{n}{2}}}e^{-\frac{|x'-y'|^2}{4t}}\left(e^{-\frac{|x_n-y_n|^2}{4t}}+e^{-\frac{|x_n+y_n|^2}{4t}}\right),\ \ x,y\in\mathbb{R}_\pm^n.
\end{align*}
For any function $f$ on $\mathbb{R}_\pm^n$, we have (\cite[Section 2.2]{DDSY})
\begin{align*}
\exp(-t\Delta_{N_\pm})f(x)=\exp(-t\Delta)f_e(x), {\rm for}\ {\rm all}\ t\geq 0\ {\rm and} \ x\in\mathbb{R}_\pm^n.
\end{align*}
And, for any function $f$ on $\mathbb{R}^n$, we have (\cite[Section 2.2]{DDSY})
\begin{align}\label{uiii}
(\exp(-t\Delta_N)f)_\pm(x)=\exp(-t\Delta_{N_\pm})f_\pm(x), {\rm for}\ {\rm all}\ t\geq 0\ {\rm and} \ x\in\mathbb{R}_\pm^n.
\end{align}
The heat kernel of $\exp(-t\Delta_N)$ is given as
\begin{align}\label{heatkernel}
p_{t,\Delta_N}(x,y)=\frac{1}{(4\pi t)^{\frac{n}{2}}}e^{-\frac{|x'-y'|^2}{4t}}\left(e^{-\frac{|x_n-y_n|^2}{4t}}+e^{-\frac{|x_n+y_n|^2}{4t}}\right)H(x_ny_n),
\end{align}
where $H:\mathbb{R}\rightarrow \{0,1\}$ is the Heaviside function defined as
$$ H(t)=\left\{\begin{array}{ll}0, &{\rm if}\ t<0,\\1, &{\rm if}\ t\geq 0.\end{array}\right.$$
%$$H(t)=0,\ {\rm if}\ t<0;\ \ \ \ H(t)=1,\ {\rm if}\ t\geq0.$$

Denote by $K_{\ell}(x,y)$ the kernel of the $\ell$-th Riesz transform $R_{N,\ell}$.  Then it was shown in \cite[Proposition 2.2]{LW} that for any $1\leq \ell\leq n-1$ and for $x,y\in\mathbb{R}^n_+$ we have:
\begin{align}\label{kernel1}
K_{\ell}(x,y)= - C_n \bigg( {x_{\ell}-y_{\ell}\over |x-y|^{n+1}} +   \frac{x_{\ell}-y_{\ell}}{(|x'-y'|^2+|x_n+y_n|^2)^{\frac{n+1}{2}}}\bigg)
\end{align}
and
\begin{align}\label{kernel2}
K_{n}(x,y)= - C_n \bigg( {x_{n}-y_{n}\over |x-y|^{n+1}} +   \frac{x_n+y_n}{(|x'-y'|^2+|x_n+y_n|^2)^{\frac{n+1}{2}}}\bigg),
\end{align}
where $C_n=\frac{\Gamma\big(\frac{n+1}{2}\big)}{\pi ^{\frac{n+1}{2}}}  $. Similar expressions also hold for $K_{\ell}(x,y) $, $\ell=1,\ldots,n$, when $x,y\in\mathbb{R}^n_-$.

From equality \eqref{heatkernel} and the formula
$$\Delta_{N}^{-\frac12}=\frac{1}{\sqrt{\pi}}\int_0^\infty e^{-t\Delta_N}\frac{dt}{\sqrt{t}},$$
one can deduce that for any $1\leq \ell\leq n$, $K_\ell(x,y)=0$ whenever $x$ and $y$ belong to distinct half-plane.
\begin{lemma}\label{CZO}
For any $\ell\in\{1,2,\ldots,n\}$, the kernel $K_{\ell}(x,y)$ satisfies the following size condition and smooth condition:
\begin{align*}
|K_{\ell}(x,y)|\leq C_{n}\frac{1}{|x-y|^{n}},
\end{align*}
and
\begin{align*}
|K_{\ell}(x,y)-K_{\ell}(x^{\prime},y)|+|K_{\ell}(y,x)-K_{\ell}(y,x^{\prime})|\leq C\frac{|x-x^{\prime}|}{|x-y|^{n+1}}
\end{align*}
for $x$, $x_{0}$, $y\in\mathbb{R}_{+}^{n}$ (or $x$, $x_{0}$, $y\in\mathbb{R}_{-}^{n}$) satisfying $|x-x^{\prime}|\leq \frac{1}{2}|x-y|$.
\end{lemma}
\begin{proof}
Consult \cite{LW} for the proof.
\end{proof}

\begin{lemma}\label{sign}
Given $\ell\in\{1,2,\ldots,n\}$, $h\in B(0,1)$ and a
system of dyadic cubes $\mathcal{D}_\pm
    := \cup_{k\in\mathbb{Z}}\mathcal{D}_{k,\pm}$ with parameter $\delta\in (0,1)$. There exists a constant  $A>0$ such that for any $Q\in \tau^h\mathcal{D}_{k,\pm}$ with center $x_{0}$ and satisfying $Q\subseteq \mathbb{R}_{\pm}^n$, one can find a ball $\hat{Q}:=B_{\mathbb{R}_{\pm}^n}(y_0,\frac{1}{12}\delta^k)\subset \mathbb{R}_{\pm}^{n}$ such that $|x_{0}-y_{0}|=A\delta^k$, and for all $(x,y)\in Q\times\hat{Q}$, $K_{\ell}(x,y)$ does not change sign and satisfies
\begin{align*}
|K_{\ell}(x,y)|\geq C\delta^{-kn}
\end{align*}
for some constant $C>0$.
\end{lemma}
\begin{proof}
To avoid confusion, we first consider the case $Q\in \tau^h\mathcal{D}_{k,+}$, which satisfies $Q\subseteq \mathbb{R}_{+}^n$.

Let $A$ be a sufficiently large number and $Q\in \tau^h\mathcal{D}_{k,+}$ be any cube  with center $x_{0}=(x^{(1)},\ldots,x^{(n)})\in\mathbb{R}_{+}^{n}$, side length $\delta^k$ and satisfying $Q\subseteq \mathbb{R}^n_+$.
For any $\ell\in\{1,2,\ldots, n\}$, we choose $y_{0}=x_{0}+A\delta^{k}e_{\ell}\in\mathbb{R}_{+}^{n}$, then
\begin{align*}
|K_{\ell}(x_{0},y_{0})|=C_{n}\left|(A\delta^{k})^{-n}+\frac{A\delta^{k}}{((A\delta^{k})^{2}+(2x^{(n)})^{2})^{\frac{n+1}{2}}}\right|\geq C_{n}A^{-n}\delta^{-kn},\ {\rm for}\ \ell\in\{1,2,...,n-1\}
\end{align*}
and
\begin{align*}
|K_{n}(x_{0},y_{0})|=C_{n}\left|(A\delta^{k})^{-n}+(2x^{(n)}+A\delta^{k})^{-n}\right|\geq A^{-n}\delta^{-kn}.
\end{align*}
%there exists a point $y_{0}\in (B(x_{0},3Ar)\backslash B(x_{0},Ar))\cap \mathbb{R}_{+}^{n}$ and a constant $c_{0}>0$ such that
%\begin{align*}
%|K(x_{0},y_{0})|\geq \frac{c_{0}}{A^{n}r^{n}}.
%\end{align*}
Let $\hat{Q}:=B_{\mathbb{R}_{+}^n}(y_0,\frac{1}{12}\delta^k)$.
By Lemma \ref{CZO}, for any $x\in Q$ and $y\in \hat{Q}$, we have
\begin{align*}
|K_{\ell}(x,y)-K_{\ell}(x_{0},y_{0})|&\leq |K_{\ell}(x,y)-K_{\ell}(x,y_{0})|+|K_{\ell}(x,y_{0})-K_{\ell}(x_{0},y_{0})|\\
&\leq C\frac{|y-y_{0}|}{|x-y|^{n+1}}+C\frac{|x-x_{0}|}{|x_{0}-y_{0}|^{n+1}}\\
%&\leq C\frac{\delta^{k}}{((A-4-c)\delta^{k})^{n+1}}+4C\frac{\delta^{k}}{(A\delta^{k})^{n+1}}\\
&\leq \frac{C_{n}}{2}A^{-n}\delta^{-kn},
\end{align*}
where in the last inequality we used the fact that $A$ is a sufficiently large constant.

If $K_{\ell}(x_{0},y_{0})>0$, then
\begin{align*}
K_{\ell}(x,y)&=K_{\ell}(x_{0},y_{0})-(K_{\ell}(x_{0},y_{0})-K_{\ell}(x,y))\geq K_{\ell}(x_{0},y_{0})-|K_{\ell}(x,y)-K_{\ell}(x_{0},y_{0})|\\
&\geq C_{n}A^{-n}\delta^{-kn}-\frac{C_{n}}{2}A^{-n}\delta^{-kn}\\
&=\frac{C_{n}}{2}A^{-n}\delta^{-kn}.
\end{align*}

If $K_{\ell}(x_{0},y_{0})<0$, then
\begin{align*}
K_{\ell}(x,y)&=K_{\ell}(x_{0},y_{0})-(K_{\ell}(x_{0},y_{0})-K_{\ell}(x,y))\leq K_{\ell}(x_{0},y_{0})+|K_{\ell}(x,y)-K_{\ell}(x_{0},y_{0})|\\
&\leq -C_{n}A^{-n}\delta^{-kn}+\frac{C_{n}}{2}A^{-n}\delta^{-kn}\\
&=-\frac{C_{n}}{2}A^{-n}\delta^{-kn}.
\end{align*}

Similarly, if $Q\in \tau^h\mathcal{D}_{k,-}$ is any cube with center $x_{0}=(x^{(1)},\ldots,x^{(n)})\in\mathbb{R}_{-}^{n}$ and side length $\delta^k$, then for any $\ell\in\{1,2,\ldots, n\}$, by choosing $y_{0}=x_{0}-A\delta^{k}e_{\ell}\in\mathbb{R}_{-}^{n}$, $\hat{Q}:=B_{\mathbb{R}_{-}^n}(y_0,\frac{1}{12}\delta^k)\subset \mathbb{R}_{-}^{n}$ and following a similar calculation as above, we can also show that $K_{\ell}(x,y)$ does not change sign for all $(x,y)\in Q\times\hat{Q}$. This ends the proof of Lemma \ref{sign}.
\end{proof}

\subsection{Besov Spaces associated with Neumann Laplacian}\label{chara}
In this subsection, we will establish several fundamental properties of the Besov space $B_{p,q}^{\alpha,\Delta_N}(\mathbb{R}^{n})$, which is useful for proving Theorem \ref{schatten} and  is also of independent interest in the theory of Besov spaces. Note that due to the lack of H\"{o}lder's continuity estimate on the whole space $\mathbb{R}^n$, the Neumann Laplacian operator does not satisfy the assumption imposed on the previous work \cite{BBD,BDY}, so many properties in the previous theory can not be applied directly to our setting. To handle this, we will investigate this Besov space by borrowing some of the ideas in \cite{CY,DDSY,LW}.
%In the last few decades, the theory of Besov spaces has been an active area of research, which, in particular, is useful to characterize the Schatten-$p$ class property of Riesz transform commutators (see for example \cite{FLLarxiv,JW,RS}). Among these works, it would be worthwhile to mention that Bui, Duong and Yan \cite{BDY} laid the foundation of the theory of Besov space associated with a certain operator $L$ under the assumption that $L$ generate an analytic semigroup $e^{-tL}$ with Gaussian upper bound on $L^2(X)$, where $X$ is a quasi-metric space of polynomial upper bounds on volumn growth (see also \cite{BBD,MR3319540,CG,GJL,Hu,LYY} for other development along this direction). However, due to the lack of H\"{o}lder's continuity estimate on the whole space $\mathbb{R}^n$, Neumann Laplacian operator does not satisfy the assumption imposed on the above work, so the previous theory can not be applied directly to our setting.

The first task is to establish a useful equivalent characterization of $B_{p,q}^{\alpha,\Delta_N}(\mathbb{R}^{n})$, which establishes the relation with the classical homogeneous Besov space.
\begin{definition}
Suppose $1\leq p,q< \infty$ and $0<\alpha<1$.  We define the (homogeneous) Besov space $B_{p,q}^{\alpha}(\mathbb{R}^{n})$ as follows:
\begin{align*}
B_{p,q}^{\alpha}(\mathbb{R}^{n})=\{f\in L_{{\rm loc}}^{1}(\mathbb{R}^{n}):\|f\|_{B_{p,q}^{\alpha}(\mathbb{R}^{n})}<\infty\},
\end{align*}
where
\begin{align*}
\|f\|_{B_{p,q}^{\alpha}(\mathbb{R}^{n})}:=\left(\int_{\mathbb{R}^{n}}\frac{\|f(\cdot+t)-f(\cdot)\|_{L^{p}(\mathbb{R}^{n})}^{q}}{|t|^{n+q\alpha}}dt\right)^{1/q}.
\end{align*}
\end{definition}
%To begin with, observe that if $f$ is a complex-valued measurable function on $\mathbb{R}^n$, which satisfies the growth condition
%\begin{align}\label{growth}
%\int_{\mathbb{R}^n}\frac{|f(x)|}{(1+|x|)^{n+\epsilon}}dx<+\infty
%\end{align}
%for some $\epsilon>0$, then by ???, the expression
%$$t\Delta_Ne^{-t\Delta_N}f(x)=\int_{\mathbb{R}^n}t\frac{\partial}{\partial t}p_t(x,y)f(y)dy$$
%is well-defined for every $x\in\mathbb{R}^n$.
%{\color{red}Definition of $B_{p,q}^{\alpha,\Delta_{N_{+}}}(\mathbb{R}_{+}^{n})$???}
Now for any $x=(x^{\prime},x_n)\in \mathbb{R}^n$, we set $\tilde{x}=(x^{\prime},-x_n)$. Let $f$ be any function defined on $\mathbb{R}^{n}_+$, its even extension $f_e$ is defined on $\mathbb{R}^n$ by
$$ f_e(x):=\left\{\begin{array}{ll}f(x), &{\rm if}\ x\in\mathbb{R}_{+}^n;\\f(\tilde{x}), &{\rm if}\ x\in\mathbb{R}_{-}^n.\end{array}\right.$$
Similarly, one can define the even extension for any function defined on $\mathbb{R}^{n}_-$.
%\begin{definition}
%Let $f$ be a function on $\mathbb{R}_\pm^n$, then $f$ is said to be $B_{p,q}^{\alpha,e}(\mathbb{R}_{\pm}^{n})$ if $f_e\in B_{p,q}^{\alpha}(\mathbb{R}^{n})$. Moreover, $B_{p,q}^{\alpha,e}(\mathbb{R}_{\pm}^{n})$ is endowed with the norm $\|f\|_{B_{p,q}^{\alpha,e}(\mathbb{R}_{\pm}^{n})}:=\|f_e\|_{B_{p,q}^{\alpha}(\mathbb{R}^{n})}$.
%\end{definition}
\begin{definition}
Suppose $1\leq p,q< \infty$ and $0<\alpha<1$. For any $f\in \mathcal{M}(\mathbb{R}^n)$,
we define its  $B_{p,q}^{\alpha,\Delta_{N_\pm}}(\mathbb{R}_\pm^{n})$ norm by the expression:
\begin{align*}
\|f\|_{B_{p,q}^{\alpha,\Delta_{N_\pm}}(\mathbb{R}_\pm^{n})}:=\left(\int_0^\infty (t^{-\alpha}\|t\Delta_{N_\pm}e^{-t\Delta_{N_\pm}}f\|_{L^p(\mathbb{R}_\pm^n)})^q\frac{dt}{t}\right)^{1/q}.
\end{align*}
\end{definition}

\begin{lemma}\label{prechar}
Suppose $1\leq p,q< \infty$ and $0<\alpha<1$, then the space $B_{p,q}^{\alpha,\Delta_{N_{\pm}}}(\mathbb{R}_{\pm}^{n})$ can be characterized in the following way:
\begin{align*}
B_{p,q}^{\alpha,\Delta_{N_{\pm}}}(\mathbb{R}_{\pm}^{n})\simeq\Big\{f\in \mathcal{M}(\mathbb{R}^n):f_{e}\in B_{p,q}^{\alpha}(\mathbb{R}^n)\Big\}.
\end{align*}
Furthermore, one has
\begin{align*}
\|f\|_{B_{p,q}^{\alpha,\Delta_{N_\pm}}(\mathbb{R}_\pm^{n})}\simeq\|f_{e}\|_{B_{p,q}^{\alpha}(\mathbb{R}^{n})}.
\end{align*}
%\begin{align*}
%B_{p,q}^{\alpha,\Delta_{N_{+}}}(\mathbb{R}_{+}^{n})=B_{p,q}^{\alpha,e}(\mathbb{R}_{+}^{n}).
%\end{align*}
\end{lemma}
\begin{proof}
By \cite[Theorem 6.5]{BBD} (see also \cite[Theorem 5.1]{BDY}), we see that
\begin{align*}
\|f\|_{B_{p,q}^{\alpha,\Delta_{N_{+}}}(\mathbb{R}_{+}^{n})}&=\left(\int_0^\infty (t^{-\alpha}\|t\Delta_{N_+} e^{-t\Delta_{N_+}}f\|_{L^p(\mathbb{R}_+^n)})^q\frac{dt}{t}\right)^{1/q}\\
&\simeq \left(\int_0^\infty (t^{-\alpha}\|t\Delta e^{-t\Delta}f_e\|_{L^p(\mathbb{R}_+^n)})^q\frac{dt}{t}\right)^{1/q}+ \left(\int_0^\infty (t^{-\alpha}\|t\Delta e^{-t\Delta}f_e\|_{L^p(\mathbb{R}_-^n)})^q\frac{dt}{t}\right)^{1/q}\\
&\simeq \left(\int_0^\infty (t^{-\alpha}\|t\Delta e^{-t\Delta}f_e\|_{L^p(\mathbb{R}^n)})^q\frac{dt}{t}\right)^{1/q}\simeq\|f_e\|_{B_{p,q}^{\alpha}(\mathbb{R}^{n})}.
\end{align*}
Similarly, one can deduce that $ \|f\|_{B_{p,q}^{\alpha,\Delta_{N_{-}}}(\mathbb{R}_{-}^{n})}\simeq \|f_e\|_{B_{p,q}^{\alpha}(\mathbb{R}^{n})}$. This ends the proof of Lemma \ref{prechar}.
%
%\begin{align*}
%\|f\|_{B_{p,q}^{\alpha,\Delta_{N_{+}}}(\mathbb{R}_{+}^{n})}
%&=\left\{\sum_{j\in\mathbb{Z}}\Big(2^{j\alpha}\|\psi_j (\sqrt{\Delta_{N_{+}}})f\|_{L^p(\mathbb{R}_{+}^n)}\Big)^{q}\right\}^{1/q}\\
%&{\color{red}\simeq \left\{\sum_{j\in\mathbb{Z}}\Big(2^{j\alpha}\|\psi_j (\sqrt{\Delta_{N}})f_e\|_{L^p(\mathbb{R}_{+}^n)}\Big)^{q}\right\}^{1/q}}\\
%&\simeq \left\{\sum_{j\in\mathbb{Z}}\Big(2^{j\alpha}\|\psi_j (\sqrt{\Delta_{N}})f_e\|_{L^p(\mathbb{R}_{+}^n)}\Big)^{q}\right\}^{1/q}+ \left\{\sum_{j\in\mathbb{Z}}\Big(2^{j\alpha}\|\psi_j (\sqrt{\Delta_{N}})f_e\|_{L^p(\mathbb{R}_{-}^n)}\Big)^{q}\right\}^{1/q}
%\end{align*}
\end{proof}

\begin{lemma}\label{Besovchar}
Suppose $1\leq p,q< \infty$ and $0<\alpha<1$, then the space $B_{p,q}^{\alpha,\Delta_{N}}(\mathbb{R}^{n})$ can be characterized in the following way:
\begin{align*}
B_{p,q}^{\alpha,\Delta_N}(\mathbb{R}^{n})=\Big\{f\in \mathcal{M}(\mathbb{R}^n):f_{+,e}\in B_{p,q}^{\alpha}(\mathbb{R}^n),f_{-,e}\in B_{p,q}^{\alpha}(\mathbb{R}^n)\Big\}.
\end{align*}
Furthermore, one has
\begin{align*}
\|f\|_{B_{p,q}^{\alpha,\Delta_N}(\mathbb{R}^{n})}\simeq\|f_{+,e}\|_{B_{p,q}^{\alpha}(\mathbb{R}^{n})}+\|f_{-,e}\|_{B_{p,q}^{\alpha}(\mathbb{R}^{n})}.
\end{align*}
\begin{proof}
By \eqref{uiii} and then Lemma \ref{prechar},
\begin{align*}
\|f\|_{B_{p,q}^{\alpha,\Delta_{N}}(\mathbb{R}^{n})}
%&=\left(\int_0^\infty (t^{-\alpha}\|t\Delta_{N}e^{-t\Delta_N}f\|_{L^p(\mathbb{R}^n)})^q\frac{dt}{t}\right)^{1/q}\\
&\simeq\left(\int_0^\infty (t^{-\alpha}\|t\Delta_{N}e^{-t\Delta_{N}}f\|_{L^p(\mathbb{R}_+^n)})^q\frac{dt}{t}\right)^{1/q}+\left(\int_0^\infty (t^{-\alpha}\|t\Delta_{N}e^{-t\Delta_{N}}f\|_{L^p(\mathbb{R}_-^n)})^q\frac{dt}{t}\right)^{1/q}\\
&=\left(\int_0^\infty (t^{-\alpha}\|t\Delta_{N_+}e^{-t\Delta_{N_+}}f_+\|_{L^p(\mathbb{R}_+^n)})^q\frac{dt}{t}\right)^{1/q}+\left(\int_0^\infty (t^{-\alpha}\|t\Delta_{N_-}e^{-t\Delta_{N_-}}f_-\|_{L^p(\mathbb{R}_-^n)})^q\frac{dt}{t}\right)^{1/q}\\
&\simeq\|f_{+,e}\|_{B_{p,q}^{\alpha}(\mathbb{R}^{n})}+\|f_{-,e}\|_{B_{p,q}^{\alpha}(\mathbb{R}^{n})}.
\end{align*}
This  ends the proof of Lemma \ref{Besovchar}.
\end{proof}
%\begin{align*}
%B_{p,q}^{\alpha,\Delta_{N}}(\mathbb{R}^{n})=\Big\{f\in L^1_{{\rm loc}}(\mathbb{R}^n):f_{+}\in B_{p,q}^{\alpha,\Delta_{N_{+}}}(\mathbb{R}_{+}^n),f_-\in B_{p,q}^{\alpha,\Delta_{N_{-}}}(\mathbb{R}_{-}^n)\Big\}.
%\end{align*}
\end{lemma}
Recall that \cite{CY} introduced the ${\rm VMO}_{\Delta_N}(\mathbb{R}^n)$ space,  the space of functions of vanishing mean oscillation associated with the semigroup $\{e^{t\Delta_N}\}_{t\geq 0}$. One has the following embedding.
\begin{coro}\label{coroo}
For any $n\geq 2$ and $1\leq p<\infty$, the following embedding holds:
\begin{align*}
B_{p,p}^{\frac{n}{p},\Delta_N}(\mathbb{R}^n)\subset {\rm VMO}_{\Delta_N}(\mathbb{R}^n).
\end{align*}
\begin{proof}
Recall from \cite{CY} that ${\rm VMO}_{\Delta_N}(\mathbb{R}^n)$ can be characterized in the following way.
$${\rm VMO}_{\Delta_N}(\mathbb{R}^n):=\{f\in\mathcal{M}(\mathbb{R}^n):f_{+,e}\in {\rm VMO}(\mathbb{R}^n)\ and\ f_{-,e}\in {\rm VMO}(\mathbb{R}^n)\}.$$
This, in combination with Lemma \ref{Besovchar} and  the classical embedding $B_{p,p}^{\frac{n}{p}}(\mathbb{R}^n)\subset {\rm VMO}(\mathbb{R}^n)$, ends the proof of Corollary \ref{coroo}.
\end{proof}
\end{coro}

Next, we will establish an interpolation theorem for the Besov space $B_{p,p}^{\alpha,\Delta_N}(\mathbb R^n)$, which plays a crucial role in the proof of sufficiency (the upper bound of Theorem \ref{schatten}). To begin with, given two Banach spaces $X_1$ and $X_2$, we set $\mathcal{F}(X_1,X_2)$ be the linear space consisting of all functions $f:\mathbb{C}\rightarrow X_1+X_2$, which are bounded and continuous on the strip $\{z\in\mathbb{C}:0\leq {\rm Re}z\leq 1\}$ and analytic on the open strip $\{z\in\mathbb{C}:0< {\rm Re}z< 1\}$, and moreover, the functions $t\rightarrow f(j+it)$ are continuous function from $t\in\mathbb{R}$ to $X_j$, $j=1,2$, which tends to zero as $|t|\rightarrow \infty$ (see \cite[Section 4.1]{Bergh}). 
%that are analytic on $\{z\in\mathbb{C}:0<{\rm Re}z<1\}$, continuous on $\{z\in\mathbb{C}:0\leq{\rm Re}z\leq1\}$ and for which the following subsets are bounded:
%\begin{align*}
%&\{f(z):z\in \mathbb{C},0<{\rm Re}z<1\}\subset X_1+X_2,\\
%&\{f(it):t\in\mathbb{R}\}\subset X_1,\\
%&\{f(1+it):t\in\mathbb{R}\}\subset X_2.
%\end{align*}
Then $\mathcal{F}(X_1,X_2)$ is a Banach space under the norm
$$\|f\|_{\mathcal{F}(X_1,X_2)}:=\max\left\{\sup\limits_{t\in\mathbb{R}}\|f(it)\|_{X_1},\sup\limits_{t\in\mathbb{R}}\|f(1+it)\|_{X_2}\right\}.$$
\begin{definition}[\cite{Bergh}]
Given two Banach spaces $X_1$ and $X_2$, we define the complex interpolation space $(X_1,X_2)_\theta$ be the linear subspace of $X_1+X_2$ consisting of all values $f(\theta)$ when $f$ varies in the preceding space of functions,
\begin{align*}
(X_1,X_2)_\theta=\left\{g\in X_1+X_2:g=f(\theta),f\in \mathcal{F}(X_1,X_2)\right\},
\end{align*}
equipped with the norm
$$\|g\|_{(X_1,X_2)_\theta}:=\inf\left\{\|f\|_{\mathcal{F}(X_1,X_2)}:g=f(\theta),f\in \mathcal{F}(X_1,X_2)\right\}.$$
\end{definition}
\begin{lemma}\label{complexinterpolation}
Let $1<p_1<p_2<\infty$, $0<\alpha<1$ and $0<\theta<1$, then the spaces $B_{p,p}^{\alpha,\Delta_N}(\mathbb R^n)$ and $(B_{p_1,p_1}^{\alpha,\Delta_N}(\mathbb R^n),B_{p_2,p_2}^{\alpha,\Delta_N}(\mathbb R^n))_{\theta_p}$  coincide, and their norms are equivalent, where $\theta_p$ satisfies $\frac{1-\theta_p}{p_1}+\frac{\theta_p}{p_2}=\frac{1}{p}$.
\end{lemma}
\begin{proof}
We first show that for any $f\in (B_{p_1,p_1}^{\alpha,\Delta_N}(\mathbb R^n),B_{p_2,p_2}^{\alpha,\Delta_N}(\mathbb R^n))_{\theta_p}$, one has
\begin{align}\label{firstincursion}
\|f\|_{B_{p,p}^{\alpha,\Delta_N}(\mathbb R^n)}\lesssim\|f\|_{(B_{p_1,p_1}^{\alpha,\Delta_N}(\mathbb R^n),B_{p_2,p_2}^{\alpha,\Delta_N}(\mathbb R^n))_{\theta_p}}.
\end{align}
To this end, for any $\xi\in \mathcal{F}(B_{p_1,p_1}^{\alpha,\Delta_N}(\mathbb R^n),B_{p_2,p_2}^{\alpha,\Delta_N}(\mathbb R^n))$ satisfying $\xi(\theta_p)=f$,  we pick $g(z)=\xi(z)_{+,e}$ and $h(z)=\xi(z)_{-,e}$, then $g(\theta_p)=\xi(\theta_p)_{+,e}=f_{+,e}$ and $h(\theta_p)=\xi(\theta_p)_{-,e}=f_{-,e}$. Besides, applying the complex interpolation theorem for the classical Besov space $B_{p,p}^{\alpha}(\mathbb R^n)$, one has
\begin{align*}
\|f\|_{B_{p,p}^{\alpha,\Delta_N}(\mathbb R^n)}
&\simeq\|f_{+,e}\|_{(B_{p_1,p_1}^{\alpha}(\mathbb R^n),B_{p_2,p_2}^{\alpha}(\mathbb R^n))_{\theta_p}}+\|f_{-,e}\|_{(B_{p_1,p_1}^{\alpha}(\mathbb R^n),B_{p_2,p_2}^{\alpha}(\mathbb R^n))_{\theta_p}}\\
&\leq \|\xi(z)_{+,e}\|_{\mathcal{F}(B_{p_1,p_1}^{\alpha}(\mathbb R^n),B_{p_2,p_2}^{\alpha}(\mathbb R^n))}+\|\xi(z)_{-,e}\|_{\mathcal{F}(B_{p_1,p_1}^{\alpha}(\mathbb R^n),B_{p_2,p_2}^{\alpha}(\mathbb R^n))}\\
&= \max\Bigg\{\sup\limits_{t\in\mathbb{R}}\|\xi(it)_{+,e}\|_{B_{p_1,p_1}^{\alpha}(\mathbb R^n)},\sup\limits_{t\in\mathbb{R}}\|\xi(1+it)_{+,e}\|_{B_{p_2,p_2}^{\alpha}(\mathbb R^n)}\Bigg\}\\
&\quad+\max\Bigg\{\sup\limits_{t\in\mathbb{R}}\|\xi(it)_{-,e}\|_{B_{p_1,p_1}^{\alpha}(\mathbb R^n)},\sup\limits_{t\in\mathbb{R}}\|\xi(1+it)_{-,e})\|_{B_{p_2,p_2}^{\alpha}(\mathbb R^n)}\Bigg\}\\
&\leq 2\max\Bigg\{\sup\limits_{t\in\mathbb{R}}\|\xi(it)\|_{B_{p_1,p_1}^{\alpha,\Delta_N}(\mathbb R^n)},\sup \limits_{t\in\mathbb{R}}\|\xi(1+it)\|_{B_{p_2,p_2}^{\alpha,\Delta_N}(\mathbb R^n)}\Bigg\}\\
&=2\|\xi\|_{\mathcal{F}(B_{p_1,p_1}^{\alpha,\Delta_N}(\mathbb R^n),B_{p_2,p_2}^{\alpha,\Delta_N}(\mathbb R^n))}.
\end{align*}
From the arbitrariness of $\xi$, we deduce inequality \eqref{firstincursion}.

Next, we show that for any $f\in B_{p,p}^{\alpha,\Delta_N}(\mathbb R^n)$, one has
\begin{align}\label{secondincursion}
\|f\|_{(B_{p_1,p_1}^{\alpha,\Delta_N}(\mathbb R^n),B_{p_2,p_2}^{\alpha,\Delta_N}(\mathbb R^n))_{\theta_p}}\lesssim\|f\|_{B_{p,p}^{\alpha,\Delta_N}(\mathbb R^n)}.
\end{align}
To this end, for any $g,h\in \mathcal{F}(B_{p_1,p_1}^{\alpha}(\mathbb R^n),B_{p_2,p_2}^{\alpha}(\mathbb R^n))$ satisfying $g(\theta_p)=f_{+,e}$ and  $h(\theta_p)=f_{-,e}$,  we pick $\xi=g\chi_{\mathbb{R}_+^n}+h\chi_{\mathbb{R}_-^n}$, then $\xi(\theta_p)=f$. Besides,
\begin{align*}
&\|\xi\|_{\mathcal{F}(B_{p_1,p_1}^{\alpha,\Delta_N}(\mathbb R^n),B_{p_2,p_2}^{\alpha,\Delta_N}(\mathbb R^n))}\\
&=\max\left\{\sup\limits_{t\in\mathbb{R}}(\|\xi(it)_{+,e}\|_{B_{p_1,p_1}^{\alpha}(\mathbb R^n)}+\|\xi(it)_{-,e}\|_{B_{p_1,p_1}^{\alpha}(\mathbb R^n)}),\sup\limits_{t\in\mathbb{R}}(\|\xi(1+it)_{+,e}\|_{B_{p_2,p_2}^{\alpha}(\mathbb R^n)}+\|\xi(1+it)_{-,e}\|_{B_{p_2,p_2}^{\alpha}(\mathbb R^n)})\right\}\\
&=\max\left\{\sup\limits_{t\in\mathbb{R}}(\|g(it)_{+,e}\|_{B_{p_1,p_1}^{\alpha}(\mathbb R^n)}+\|h(it)_{-,e}\|_{B_{p_1,p_1}^{\alpha}(\mathbb R^n)}),\sup\limits_{t\in\mathbb{R}}(\|g(1+it)_{+,e}\|_{B_{p_2,p_2}^{\alpha}(\mathbb R^n)}+\|h(1+it)_{-,e}\|_{B_{p_2,p_2}^{\alpha}(\mathbb R^n)})\right\}\\
&\lesssim\max\left\{\sup\limits_{t\in\mathbb{R}}(\|g(it)\|_{B_{p_1,p_1}^{\alpha}(\mathbb R^n)}+\|h(it)\|_{B_{p_1,p_1}^{\alpha}(\mathbb R^n)}),\sup\limits_{t\in\mathbb{R}}(\|g(1+it)\|_{B_{p_2,p_2}^{\alpha}(\mathbb R^n)}+\|h(1+it)\|_{B_{p_2,p_2}^{\alpha}(\mathbb R^n)})\right\}\\
&\lesssim\max\left\{\sup\limits_{t\in\mathbb{R}}\|g(it)\|_{B_{p_1,p_1}^{\alpha}(\mathbb R^n)},\sup\limits_{t\in\mathbb{R}}\|g(1+it)\|_{B_{p_2,p_2}^{\alpha}(\mathbb R^n)}\right\}+\max\left\{\sup\limits_{t\in\mathbb{R}}\|h(it)\|_{B_{p_1,p_1}^{\alpha}(\mathbb R^n)},\sup\limits_{t\in\mathbb{R}}\|h(1+it)\|_{B_{p_2,p_2}^{\alpha}(\mathbb R^n)}\right\}.
\end{align*}
From the arbitrariness of $g$ and $h$, we conclude that
\begin{align*}
\|f\|_{(B_{p_1,p_1}^{\alpha,\Delta_N}(\mathbb R^n),B_{p_2,p_2}^{\alpha,\Delta_N}(\mathbb R^n))_{\theta_p}}
&\lesssim \|\xi\|_{\mathcal{F}(B_{p_1,p_1}^{\alpha,\Delta_N}(\mathbb R^n),B_{p_2,p_2}^{\alpha,\Delta_N}(\mathbb R^n))}\\
&\lesssim \|f_{+,e}\|_{(B_{p_1,p_1}^{\alpha}(\mathbb R^n),B_{p_2,p_2}^{\alpha}(\mathbb R^n))_{\theta_p}}+\|f_{-,e}\|_{(B_{p_1,p_1}^{\alpha}(\mathbb R^n),B_{p_2,p_2}^{\alpha}(\mathbb R^n))_{\theta_p}}\\
&\lesssim \|f\|_{B_{p,p}^{\alpha,\Delta_N}(\mathbb R^n)}.
\end{align*}
This ends the proof of Lemma \ref{complexinterpolation}.
\end{proof}

%Suppose $1\leq p,q< \infty$ and $0<\alpha<1$. Let $f\in L_{{\rm loc}}^{1}(\mathbb{R}^{n})$. Then we say that $f$ belongs to Besov space $B_{p,q}^{\alpha}(\mathbb{R}^{n})$ if
%\begin{align*}
%\int_{\mathbb{R}^{n}}\frac{\|f(\cdot+t)-f(x)\|_{L^{p}(\mathbb{R}^{n})}^{q}}{|t|^{n+q\alpha}}dt<\infty.
%\end{align*}

\section{Theorem \ref{schatten}: The case $p>n$}\label{three}
\subsection{Proof of the Necessary Condition}
The main task in this section is to show that $b\in B_{p,p}^{\frac{n}{p},\Delta_N}(\mathbb{R}^{n})$ provided  $[b,R_{N,\ell}]\in S^{p}$ for some $p>n$.

Given $h\in B(0,1)$ and a
system of dyadic cubes $\mathcal{D}_{\pm}
    := \cup_{k\in\mathbb{Z}}\mathcal{D}_{k,\pm}$ with parameter $\delta\in (0,1)$,
%for any $\nu\in\mathbb{R}_{\pm}^{n}$, we let $\mathbb{R}_{\pm,\nu}^{n}$ be the translation of half-plane in the $\nu$ direction, that is,
%$$\mathbb{R}_{\pm,\nu}^{n}:=\{x+\nu:x\in\mathbb{R}_{\pm}^{n}\}.$$
%{\color{red}(remark)}Then we let $\mathcal{D}_{k,+}^{\nu}$ (resp. $\mathcal{D}_{k,-}^{\nu}$) be the dyadic partition of $\mathbb{R}_{+,\nu}^{n}$ (resp. $\mathbb{R}_{-,\nu}^{n}$) into standard dyadic cubes $Q$ of the form $[(a_{1}+\nu_{1})2^{k},(a_{1}+\nu_{1}+1)2^{k})\times [(a_{2}+\nu_{2})2^{k},(a_{2}+\nu_{2}+1)2^{k})\times\ldots\times [(a_{n}+\nu_{n})2^{k},(a_{n}+\nu_{n}+1)2^{k}) $, where each $a_{i}$ is a non-negative (resp. non-positive) integer, $i=1,2,\ldots,n$. Next, let $\mathcal{D}_{\pm}^{\nu}:=\bigcup_{k\in\mathbb{Z}}\mathcal{D}_{k,\pm}^{\nu}$ be a system of dyadic cubes such that for any $k\in\mathbb{Z}$,
%$$\mathbb{R}_{\pm,\nu}^{n}=\bigcup_{Q\in\mathcal{D}_{k,\pm}^{\nu}}Q.$$
%we define the conditional expectation of a locally integrable function $f$ on $\mathbb{R}_{\pm}^{n}$  with respect to the increasing family of $\sigma-$algebras $\sigma(\mathcal{D}_{k,\pm})$   by the expression: $$E_{k,\pm}(f)(x)=\sum_{Q\in \mathcal{D}_{k,\pm}}(f)_{Q}\chi_{Q}(x),\ x\in\mathbb{R}_{\pm}^n.$$
we  define the conditional expectation of a locally integrable function $f$ on $\mathbb{R}^{n}$  with respect to the increasing family of $\sigma-$algebras $\sigma(\tau^h\mathcal{D}_{k})$   by the expression: $$E_{k,h}(f)(x)=\sum_{Q\in \tau^h\mathcal{D}_{k}}(f)_{Q}\chi_{Q}(x),\ x\in\mathbb{R}^n,$$
where we denote $(f)_{Q}$ be the average of $f$ over $Q$, that is, $(f)_{Q}:=\fint_{Q}f(x)dx:=\frac{1}{|Q|}\int_{Q}f(x)dx$, and where we denote the translated system of standard dyadic cubes by $\tau^h\mathcal{D}_{k}=\{\tau^hQ\}_{Q\in \mathcal{D}_{k}}$. For simplicity, we set $E_k(f)(x):=E_{k,0}(f)(x)$, then it can be verified directly that for any $k\in\mathbb{Z}$ and $h\in\mathbb{R}^n$,
$$E_{k}(\tau^hf)(x)=\tau^h E_{k,h}(f)(x).$$

For any $Q\in \tau^h\mathcal{D}_{k}$, we let $h_{Q}^{1}$, $h_{Q}^{2},\ldots, h_{Q}^{M_Q-1}$ be a family of Haar functions associated to $Q$. Next, we choose $h_{Q}$ among these Haar functions such that $\left|\int_{Q}b(x)h_{Q}^{\epsilon}(x)\,dx\right|$ is maximal with respect to $\epsilon=1,2,\ldots,M_Q-1$. Note that the function $(E_{k+1,h}(b)(x)-E_{k,h}(b)(x))\chi_Q(x)$ is a sum of $M_Q-1$ Haar functions. That is, we are in a finite dimensional setting and all $L^p$-spaces have comparable norms. So
we have that
\begin{align}\label{tttt1}
\left(\fint_{Q}|E_{k+1,h}(b)(x)-E_{k,h}(b)(x)|^{p}\,dx\right)^{1/p}
%&\leq C|T|^{-1/p}\left\|\sum_{\epsilon=1}^{M_n-1}\langle f,h_{T}^{\epsilon} \rangle h_{T}^{\epsilon}\right\|_{p}\nonumber\\
%&\leq C|T|^{-1/p}\left\|\left(\sum_{k\in\mathbb{Z}}\left|\sum_{P\in\tile_{-k+1}}\sum_{\nu=1}^{M_{n}-1}\left\langle \sum_{\epsilon=1}^{M_n-1}\langle f,h_{T}^{\epsilon} \rangle h_{T}^{\epsilon},h_{P}^{\nu}\right\rangle h_{P}^{\nu}\right|^{2}\right)^{1/2}\right\|_{p}\nonumber\\
%&\leq C|T|^{-1/p}\left\|\sum_{\epsilon=1}^{M_n-1}\langle f,h_T^{\epsilon} \rangle h_T^{\epsilon}\right\|_{p}\nonumber\\
&\leq C  |Q|^{-1/2}\left|\int_{Q}b(x)h_{Q}(x)\,dx\right|,
\end{align}
where $C$ is a constant only depending on $p$ and $n$.

\begin{lemma}\label{step1}
Given a
system of dyadic cubes $\mathcal{D}_{\pm}
    := \cup_{k\in\mathbb{Z}}\mathcal{D}_{k,\pm}$ with parameter $\delta\in (0,1)$. Let $1< p<\infty$ and suppose that $b\in \mathcal{M}(\mathbb{R}^n)$ satisfying $\|[b,R_{N,\ell}]\|_{S^p}<\infty$ for some $\ell\in\{1,2,\ldots,n\}$, then there exists a constant $C>0$ such that
\begin{align} \label{e:step1}
\sup\limits_{h\in B(0,1)}\sum_{k\in\mathbb{Z}}\sum_{\substack{Q\in\tau^h\mathcal{D}_{k,\pm}\\Q\subseteq \mathbb{R}_\pm^n}}\fint_{Q}|E_{k+1,h}(b)(x)-E_{k,h}(b)(x)|^{p}dx \leq C \|[b,R_{N,\ell}]\|_{S^p} ^{p}.
\end{align}
\end{lemma}
\begin{proof}
To begin with,
by \eqref{tttt1}, we have
\begin{align}\label{comcom1}
%\delta^{-nk}\int_{\mathbb{R}_{\pm}^n}|E_{k+1,\pm}(b)(x)-E_{k,\pm}(b)(x)|^{p}dx
\sum_{\substack{Q\in\tau^h\mathcal{D}_{k,\pm}\\Q\subseteq \mathbb{R}_\pm^n}}\fint_{Q}|E_{k+1,h}(b)(x)-E_{k,h}(b)(x)|^{p}dx\leq C\sum_{\substack{Q\in\tau^h\mathcal{D}_{k,\pm}\\Q\subseteq \mathbb{R}_\pm^n}}|Q|^{-p/2}\left|\int_{Q}b(x)h_{Q}(x)dx\right|^{p}.
\end{align}

To continue, for any $Q\in\tau^h\mathcal{D}_{k,\pm}$ satisfying $Q\subseteq \mathbb{R}^n_\pm$, let $\hat{Q}$ be the ball chosen in Lemma \ref{sign}, then $K_{\ell}(x,y)$ does not change sign for all $(x,y)\in Q\times \hat{Q}$ and
\begin{align}\label{lower}
|K_{\ell}(x,y)|\geq\frac{C}{|Q|},
\end{align}
for some constant $C>0$.
Now for any set $S$ and function $f$, we define $\alpha_{S}(f)$ be the median value of $f$ over $S$, which means $\alpha_{S}(f)$ is a real number such that
\begin{align*}
\left|\left\{x\in S:f(x)>\alpha_{S}(f)\right\}\right|\leq\frac{1}{2}|S|\ \ {\rm and}\ \ \left|\left\{x\in S:f(x)<\alpha_{S}(f)\right\}\right|\leq\frac{1}{2}|S|.
\end{align*}
A median value always exists, but  may not be unique (see for example \cite{Journe}). With this notation, we
  denote
\begin{align}\label{e:E1S}
E_{1}^{Q}:=\left\{x\in Q:b(x)\leq\alpha_{\hat{Q}}(b)\right\}\ \ {\rm and}\ \
E_{2}^{Q}:=\left\{x\in Q:b(x)>\alpha_{\hat{Q}}(b)\right\}.
\end{align}
%\begin{align} \label{e:E1S}
%E_{1}^{S}:=\left\{x\in S:b(x) < \alpha_{S}(b)\right\}\ \ {\rm and}\ \
%E_{2}^{S}:=\left\{x\in S:b(x)>\alpha_{S}(b)\right\},
%\end{align}
%we have, with $ S= \hat Q $,   the upper bound $  \lvert  E ^{\hat Q} _{j}\rvert \leq \tfrac{1}2 \lvert  \hat  Q\rvert  $ for $ j=1,2$.

Next we decompose $Q$ into a union of sub-cubes by writing $Q=\bigcup_{i=1}^{M_Q}P_{i}$, where $P_{i}\in\mathcal{D}_{k+1,\pm}$ and $P_{i}\subseteq Q$ satisfying $P_{i}\neq P_{j}$ if $i\neq j$.
By the cancellation property of $h_{Q,\pm}$, we see that
\begin{align}\label{comcom2}
|Q|^{-1/2}\left|\int_{Q}b(x)h_{Q}(x)dx\right|&=|Q|^{-1/2}\left|\int_{Q}(b(x)-\alpha_{\hat{Q}}(b))h_{Q}(x)\,dx\right|\nonumber\\
&\leq \frac{1}{|Q|}\int_{Q}\left|b(x)-\alpha_{\hat{Q}}(b)\right|dx\nonumber\\
&\leq \frac{1}{|Q|}\sum_{i=1}^{M_Q}\int_{P_{i}}\left|b(x)-\alpha_{\hat{Q}}(b)\right|dx\nonumber\\
&\leq \frac{1}{|Q|}\sum_{i=1}^{M_Q}\int_{P_{i}\cap E_{1}^{Q}}\left|b(x)-\alpha_{\hat{Q}}(b)\right|dx+ \frac{1}{|Q|}\sum_{i=1}^{M_Q}\int_{P_{i}\cap E_{2}^{Q}}\left|b(x)-\alpha_{\hat{Q}}(b)\right|dx\nonumber\\
&=:{\rm I}_{1}^{Q}+{\rm I}_{2}^{Q}.
\end{align}
%Above, we are using the notation \eqref{e:E1S}.
Now we denote
\begin{align*}
F_{1}^{Q}:=\{{\hat y}\in \hat{Q}:b({\hat y})\geq\alpha_{\hat{Q}}(b)\}\ \ {\rm and}\ \
F_{2}^{Q}:=\{{\hat y}\in \hat{Q}:b({\hat y})\leq\alpha_{\hat{Q}}(b)\}.
\end{align*}
Then by the definition of $\alpha_{\hat{Q}}(b)$, we have $|F_{1}^{Q}|=|F_{2}^{Q}|\sim|\hat{Q}|$ and $F_{1}^{Q}\cup F_{2}^{Q}=\hat{Q}$. Note that for $s=1,2$, if $x\in E_{s}^{Q}$ and $y\in F_{s}^{Q}$, then
\begin{align*}
\left|b(x)-\alpha_{\hat{Q}}(b)\right|&\leq\left|b(x)-\alpha_{\hat{Q}}(b)\right|+\left|\alpha_{\hat{Q}}(b)-b(y)\right|\\
&=\left|b(x)-\alpha_{\hat{Q}}(b)+\alpha_{\hat{Q}}(b)-b(y)\right|= \left|b(y)-b(x)\right|.
\end{align*}
Therefore, for $ s=1,2$,
\begin{align}\label{haha}
{\rm I}_{s}^{Q}&\lesssim \frac{1}{|Q|}\sum_{i=1}^{M_Q}\int_{P_{i}\cap E_{s}^{Q}}\left|b(x)-\alpha_{\hat{Q}}(b)\right|dx\frac{|F_{s}^{Q}|}{|Q|}\nonumber\\
&\lesssim \frac{1}{|Q|}\sum_{i=1}^{M_Q}\int_{P_{i}\cap E_{s}^{Q}}\int_{F_{s}^{Q}}\left|b(x)-\alpha_{\hat{Q}}(b)\right|\left|K_{\ell}(x,y)\right|dydx\nonumber\\
&\lesssim \frac{1}{|Q|}\sum_{i=1}^{M_Q}\int_{P_{i}\cap E_{s}^{Q}}\int_{F_{s}^{Q}}\left|b(y)-b(x)\right|\left|K_{\ell}(x,y)\right|dydx\nonumber\\
&=\frac{1}{|Q|}\sum_{i=1}^{M_Q}\left|\int_{P_{i}\cap E_{s}^{Q}}\int_{F_{s}^{Q}}(b(y)-b(x))K_{\ell}(x,y)dydx\right|,
\end{align}
where in the last equality we used the fact that $K_{\ell}(x,y)$ and $b(y)-b(x)$ do not  change sign for $(x,y)\in (P_{i}\cap E_{s}^{Q})\times F_{s}^{Q}$, $s=1,2$. This, in combination with \eqref{comcom1} and \eqref{comcom2}, implies that
\begin{align}\label{eee ortho S norm}
%\delta^{-nk}\int_{\mathbb{R}_{\pm}^n} &|E_{k+1,\pm}(b)(x)-E_{k,\pm}(b)(x)|^{p}dx
%\nonumber\\
\sum_{\substack{Q\in\tau^h\mathcal{D}_{k,\pm}\\Q\subseteq \mathbb{R}_\pm^n}}|Q|^{-p/2}\left|\int_{Q}b(x)h_{Q}(x)dx\right|^{p}
& \lesssim  \sum_{s=1}^{2}\sum_{\substack{Q\in\tau^h\mathcal{D}_{k,\pm}\\Q\subseteq \mathbb{R}_\pm^n}}\left|{\rm I}_{s}^{Q}\right|^{p}
\nonumber\\
&\lesssim  \sum_{s=1}^{2}\sum_{\substack{Q\in\tau^h\mathcal{D}_{k,\pm}\\Q\subseteq \mathbb{R}_\pm^n}}\left(\sum_{i=1}^{M_Q}\left|\left\langle[b,R_{N,\ell}] \frac{|P_{i}|^{1/2}
\chi_{F_{s}^{Q}}}{|Q|},\frac{\chi_{P_i\cap E_{s}^{Q}}}{|P_{i}|^{1/2}}\right\rangle\right|\right)^{p}.
\end{align}
Note that $e_Q:=\frac{|P_{i}|^{1/2}
\chi_{F_{s}^{Q}}}{|Q|} \subset \hat Q$ and $f_Q:=\frac{\chi_{P_i\cap E_{s}^{Q}}}{|P_{i}|^{1/2}} \subset Q$ satisfy
$|e_Q|, |f_Q|\leq C|Q|^{-{1\over2}}\chi_{cQ} $, where $C$ and $c$ are absolute constants independent of $Q$.
Summing this last inequality over $ k\in \mathbb Z $ and then applying \eqref{e:NWO}, we complete the proof of   Lemma \ref{step1}.
\end{proof}
\begin{coro}\label{direct}
Given a
system of dyadic cubes $\mathcal{D}_{\pm}
    := \cup_{k\in\mathbb{Z}}\mathcal{D}_{k,\pm}$ with parameter $\delta\in (0,1)$. Let $1< p<\infty$ and suppose that $b\in \mathcal{M}(\mathbb{R}^n)$ satisfying $\|[b,R_{N,\ell}]\|_{S^p}<\infty$ for some $\ell\in\{1,2,\ldots,n\}$, then there exists a constant $C>0$ such that
\begin{align}
\sum_{k\in\mathbb{Z}}\delta^{-nk}\| E _{k+1,\pm}(b) -E_{k,\pm}(b)\|_{L^p(\mathbb{R}_{\pm}^{n})} ^{p} \leq C \|[b,R_{N,\ell}]\|_{S^p} ^{p}.
\end{align}
\end{coro}
\begin{proof}
It is a direct consequence of Lemma \ref{step1} by taking $h=0$ and noting that for any $x\in Q\in \mathcal{D}_{k,\pm}$, one has $E _{k,\pm}(b)(x)=E _{k,0}(b)(x)$. This ends the proof of Corollary \ref{direct}.
%\begin{align}
%\delta^{-nk}\int_{\mathbb{R}_{\pm}^n}|E_{k+1,\pm}(b)(x)-E_{k,\pm}(b)(x)|^{p}dx
%&=\sum_{Q\in\mathcal{D}_{k,\pm}}\fint_{Q}|E_{k+1,\pm}(b)(x)-E_{k,\pm}(b)(x)|^{p}dx\nonumber\\
%&=\sum_{Q\in\mathcal{D}_{k,\pm}}\fint_{Q}|E_{k+1,0}(b)(x)-E_{k,0}(b)(x)|^{p}dx .
%\end{align}
\end{proof}

\begin{coro}\label{p:step1}
Given a
system of dyadic cubes $\mathcal{D}_{\pm}
    := \cup_{k\in\mathbb{Z}}\mathcal{D}_{k,\pm}$ with parameter $\delta\in (0,1)$. Let $1<p<\infty$ and suppose that $b\in \mathcal{M}(\mathbb{R}^n)$ satisfying $\|[b,R_{N,\ell}]\|_{S^p}<\infty$  for some $\ell\in\{1,2,\ldots,n\}$, then there exists a constant $C>0$ such that for any $k\in\mathbb{Z}$,
\begin{align} \label{e:step11}
\|b-E_{k,\pm}(b)\|_{L^p(\mathbb{R}_{\pm}^{n})}\leq C\delta^{nk/p}\|[b,R_{N,\ell}]\|_{S^p}.
\end{align}
\end{coro}

\begin{proof}
It follows from Lemma \ref{thm:convergence} that $E_{k,\pm}(b)\rightarrow b$ a.e. as $k\rightarrow \infty$. By Corollary~\ref{direct}, $\|E_{k+1,\pm}(b)-E_{k,\pm}(b)\|_{L^p(\mathbb{R}_{\pm}^{n})}\leq C\delta^{nk/p}\|[b,R_{N,\ell}]\|_{S^{p}}$.  Combining these two facts and summing the geometric series yield the conclusion.
\end{proof}

\begin{lemma}\label{step2}
Given a
system of dyadic cubes $\mathcal{D}_{\pm}
    := \cup_{k\in\mathbb{Z}}\mathcal{D}_{k,\pm}$ with parameter $\delta\in (0,1)$. Let $1<p<\infty$ and suppose that $b\in \mathcal{M}(\mathbb{R}^n)$ satisfying $\|[b,R_{N,\ell}]\|_{S^p}<\infty$  for some $\ell\in\{1,2,\ldots,n\}$, then there exists a constant $C>0$ such that
\begin{align}\label{eee Besov type}
\left(\sum_{k\in\mathbb{Z}} \delta^{-nk}\|b-E_{k,\pm}(b)\|_{L^p(\mathbb{R}_{\pm}^{n})}^{p}\right)^{1/p}&\leq C  \|[b,R_{N,\ell}]\|_{S^{p}}.
\end{align}
\end{lemma}
\begin{proof}
It suffices to show that
\begin{align}
\left(\sum_{k=L}^M\delta^{-nk}\|b-E_{k,\pm}(b)\|_{L^p(\mathbb{R}_{\pm}^{n})}^{p}\right)^{1/p}&\leq C  \|[b,R_{N,\ell}]\|_{S^{p}}
\end{align}
for some constant $C>0$ independent of $L<M\in\mathbb{N}$. To this end, we denote the term in the left hand side above by $\mathfrak J$ and then
note that
\begin{align*}
\mathfrak J&\leq \left(\sum_{k=L}^M \delta^{-nk}\|b-E_{k+1,\pm}(b)\|_{L^p(\mathbb{R}_{\pm}^{n})}^{p}\right)^{1/p}+\left(\sum_{k=L}^M \delta^{-nk}\|E_{k+1,\pm}(b)-E_{k,\pm}(b)\|_{L^p(\mathbb{R}_{\pm}^{n})}^{p}\right)^{1/p}\\
&=\left(\sum_{k=L+1}^{M+1}\delta^{-n(k-1)}\|b-E_{k,\pm}(b)\|_{L^p(\mathbb{R}_{\pm}^{n})}^{p}\right)^{1/p}+\left(\sum_{k=L}^M\delta^{-nk}\|E_{k+1,\pm}(b)-E_{k,\pm}(b)\|_{L^p(\mathbb{R}_{\pm}^{n})}^{p}\right)^{1/p}\\
%&\leq \delta^{n/p}\left(\sum_{k\in\mathbb{Z}}\delta^{-nk}\|b-E_{k,\pm}^\nu(b)\|_{L^p(\mathbb{R}_{\pm}^{n})}^{p}\right)^{1/p}
%+(2n+1)^{(2n+2)M/p}\|b-E_{M+1}(b)\|_{p}\\
%&\quad
%+\left(\sum_{k\in\mathbb{Z}}\delta^{-nk}\|E_{k+1,\pm}^\nu(b)-E_{k,\pm}^\nu(b)\|_{L^p(\mathbb{R}_{\pm}^{n})}^{p}\right)^{1/p}\\
&=: {\textup{Term}_{1}}+{\textup{Term}_{2}}. %+{\textup{Term}_{3}}.
\end{align*}
To continue, we first note that Corollary~\ref{direct} controls $ {\textup{Term}_{2}}$.  ${\textup{Term}_{1}}$ is dominated by
\begin{align}\label{righ}
\delta^{n/p}\left(\sum_{k=L}^{M}\delta^{-nk}\|b-E_{k,\pm}(b)\|_{L^p(\mathbb{R}_{\pm}^{n})}^{p}\right)^{1/p}+\delta^{-nM/p}\|b-E_{M+1,\pm}(b)\|_{L^p(\mathbb{R}_{\pm}^{n})}.
\end{align}
By Corollary \ref{p:step1}, we see that the first term of the right-hand side in \eqref{righ} can be absorbed into $ \mathfrak J$, while the second term can be dominated by $C\|[b,R_{N,\ell}]\|_{S^{p}}$.
 %Proposition~\ref{p:step1} controls $ {\textup{Term}_{2}}$, and
This ends the proof of Lemma \ref{step2}.
  \end{proof}

\begin{proposition}\label{schattenlarge1}
Let $1<p<\infty$ and suppose that $b\in \mathcal{M}(\mathbb{R}^n)$ satisfying $\|[b,R_{N,\ell}]\|_{S^p}<\infty$  for some $\ell\in\{1,2,\ldots,n\}$, then there exists a constant $C>0$ such that
\begin{align*}
\|b\|_{B_{p,p}^{\frac{n}{p},\Delta_N}(\mathbb{R}^n)}\leq C\|[b,R_{N,\ell}]\|_{S^p}.
\end{align*}
\end{proposition}
\begin{proof}
By Lemma \ref{Besovchar}, it suffices to show that
\begin{align}\label{justify}
\int_{\mathbb{R}_{+}^{n}}\int_{\mathbb{R}_{+}^{n}}\frac{|b(x)-b(y)|^{p}}{|x-y|^{2n}}dxdy\lesssim \|[b,R_{N,\ell}]\|_{S^{p}}^p
\end{align}
and  that
\begin{align}\label{justify2}
\int_{\mathbb{R}_{-}^{n}}\int_{\mathbb{R}_{-}^{n}}\frac{|b(x)-b(y)|^{p}}{|x-y|^{2n}}dxdy\lesssim \|[b,R_{N,\ell}]\|_{S^{p}}^p.
\end{align}

To this end, we note that
\begin{align}
\int_{\mathbb{R}^{n}_{\pm}}\int_{\mathbb{R}^{n}_{\pm}}\frac{|b(x)-b(y)|^{p}}{|x-y|^{2n}}dxdy\leq C \sum_{k\in\mathbb{Z}}\delta^{-2nk}\iint_{\substack{|x-y|\leq \delta^{k+1}\\x,y\in\mathbb{R}_{\pm}^n}}|b(x)-b(y)|^{p}dxdy.
\end{align}
%Hence, it suffices to show that
%\begin{align}\label{maingoal}
%\sum_{k=L}^{M}\delta^{-2nk}\iint_{\substack{|x-y|\leq \delta^{k+1}\\x,y\in\mathbb{R}_{\pm}^n}}|b(x)-b(y)|^{p}dxdy\leq C\|[b,R_{N,\ell}]\|_{S^{p}}^p,
%\end{align}
%where $C$ is a constant independent of $L<M\in\mathbb{Z}$.
To continue, we first recall that there exists a collection $\{\mathcal{D}_\pm^\nu\colon
    \nu = 1,2,\ldots ,\kappa\}$ of adjacent systems of dyadic cubes on $\mathbb{R}_\pm^n$ with
    parameters $\delta\in (0, \frac{1}{96}) $ and $C_{adj} := 8\delta^{-3}$ such that the properties in Section \ref{s2} hold. With these collection of adjacent systems, we  define the conditional expectation of a locally integrable function $f$ on $\mathbb{R}_\pm^n$  with respect to the increasing family of $\sigma-$algebras $\sigma(\mathcal{D}_{k,\pm}^\nu)$   by the expression: $$E_{k,\pm}^\nu(f)(x)=\sum_{Q\in \mathcal{D}_{k,\pm}^\nu}(f)_{Q}\chi_{Q}(x),\ x\in\mathbb{R}_{\pm}^n.$$
Next, we note that there is an absolute constant $k_0>0$ such that for any $Q\in \mathcal{D}_{k,\pm}^{0}$, the standard dyadic partition of $\mathbb{R}^n_{\pm}$ defined in Section \ref{s2},  one has
$$Q_{\delta^{k+1}}:=\{y\in\mathbb{R}_{\pm}^n:d(y,Q)\leq \delta^{k+1}\}\subset B_{\mathbb{R}_{\pm}^{n}}(x_0,\delta^{k-k_0}),$$
where we used the notation $d(y,Q)$ to denote the distance from a point $y\in\mathbb{R}_{\pm}^{n}$ to a set $Q$ and the notation $x_0$ to denote the centre of $Q$.
Next, we apply Lemma \ref{thm:existence2} to see that there exist $\nu \in \{1, 2, \ldots, \kappa\}$ and
$Q^{\prime}\in\mathcal{D}_{k-k_0-2,\pm}^\nu$  such that
\begin{equation}\label{eq:ball;included}
    B_{\mathbb{R}_{\pm}^n}(x,\delta^{k-k_0})\subseteq Q^{\prime}\subseteq B_{\mathbb{R}_{\pm}^n}(x,C_{adj}\delta^{k-k_0}).
\end{equation}
With these observations, we conclude that
\begin{align*}
\sum_{k\in\mathbb{Z}}\delta^{-2nk}\iint_{\substack{|x-y|\leq \delta^{k+1}\\x,y\in\mathbb{R}_{\pm}^n}}|b(x)-b(y)|^{p}dxdy
&\lesssim \sum_{k\in\mathbb{Z}}\delta^{-2nk}\sum_{Q\in\mathcal{D}_{k,\pm}^{0}}\int_Q\int_{Q_{\delta^{k+1}}}|b(x)-b(y)|^{p}dxdy\\
&\lesssim \sum_{k\in\mathbb{Z}}\delta^{-2nk}\sum_{Q\in\mathcal{D}_{k,\pm}^{0}}\int_{Q^\prime}\int_{Q^\prime}|b(x)-b(y)|^{p}dxdy\\
&\lesssim \sum_{\nu=1}^{\kappa} \sum_{k\in\mathbb{Z}}\delta^{-2nk}\sum_{Q\in\mathcal{D}_{k-k_0-2,\pm}^{\nu}}\int_{Q}\int_{Q}|b(x)-b(y)|^{p}dxdy,
\end{align*}
where in the last inequality we used the fact that each $Q^{\prime}\in \mathcal{D}_{k-k_0-2,\pm}^{\nu}$ contains at most an absolute constant of the $Q\in\mathcal{D}_{k,\pm}^{0}$. Next, we note that the right hand side above is dominated by
\begin{align*}
&\sum_{\nu=1}^{\kappa} \sum_{k\in\mathbb{Z}}\delta^{-2nk}\sum_{Q\in\mathcal{D}_{k-k_0-2,\pm}^{\nu}}\int_{Q}\int_{Q}|b(x)-E_{k-k_{0}-2,\pm}^{\nu}b(x)|^{p}+|b(y)-E_{k-k_{0}-2,\pm}^{\nu}b(y)|^{p}dxdy\\
&=\sum_{\nu=1}^{\kappa} \sum_{k\in\mathbb{Z}}\delta^{-nk}\sum_{Q\in\mathcal{D}_{k-k_0-2,\pm}^{\nu}}\int_{Q}|b(x)-E_{k-k_{0}-2,\pm}^{\nu}b(x)|^{p}dx\\
&\leq \sum_{\nu=1}^{\kappa} \sum_{k\in\mathbb{Z}} \delta^{-nk}\|b-E_{k,\pm}^\nu(b)\|_{L^p(\mathbb{R}_{\pm}^{n})}^{p}.
\end{align*}
Applying Lemma \ref{step2} to the last term, we deduce inequalities \eqref{justify} and \eqref{justify2}, and therefore, the proof of Proposition \ref{schattenlarge1} is complete.
\end{proof}
\subsection{Proof of the Sufficient Condition}\label{s:Sufficient}
%Now we can provide a proof of the sufficient condition.
%\begin{definition}
%We say that a function $f\in BMO_{\Delta_N}(\mathbb{R}^n)$ belongs to $VMO_{\Delta_N}(\mathbb{R}^n)$, the space of functions of vanishing mean oscillation associated with the semigroup $\{e^{-t\Delta_N}\}_{t\geq 0}$, if it satisfies the following conditions:
%\begin{align*}
%&\lim_{r\rightarrow 0}\sup\limits_{B\subset \mathbb{R}^n:r_B\leq r}\left(\fint_B|f(x)-e^{r_B^2\Delta_N}f(x)|^2dx\right)^{1/2}=0,\\
%&\lim_{r\rightarrow \infty}\sup\limits_{B\subset \mathbb{R}^n:r_B\geq r}\left(\fint_B|f(x)-e^{r_B^2\Delta_N}f(x)|^2dx\right)^{1/2}=0,\\
%&\lim_{r\rightarrow \infty}\sup\limits_{B\subset B(0,r)^c}\left(\fint_B|f(x)-e^{r_B^2\Delta_N}f(x)|^2dx\right)^{1/2}=0.
%\end{align*}
%\end{definition}

\begin{proposition}\label{schattenlarge2}
Suppose $\ell\in \{1,2,\ldots,n\}$, $n<p<\infty$ and $b\in  \mathcal{M}(\mathbb{R}^n)$. If  $b\in B_{p,p}^{\frac{n}{p},\Delta_N}(\mathbb{R}^n)$, then $[b,R_{N,\ell}]\in S^p$.
\end{proposition}
\begin{proof}
  We will modify the proof in \cite{JW},  which depends on general estimates for Schatten norms of integral operators.
To begin with, we recall from Corollary \ref{coroo} that $B_{p,p}^{\frac{n}{p},\Delta_N}(\mathbb{R}^n)\subset {\rm VMO}_{\Delta_N}(\mathbb{R}^n)$. Then, the condition $b\in B_{p,p}^{\frac{n}{p},\Delta_N}(\mathbb{R}^n)$ guarantees the compactness of $[b,R_{N,\ell}]$ (see \cite[Theorem 5.3]{CY}). Next,
recall that Russo (\cite{Russo})
   showed that for a general measure space $(X,\mu)$, if $p>2$ and $K(x,y)\in  L^{2}(X\times X)$, then the integral operator $T$ associated to the kernel $K(x,y)$ satisfies the following estimate:
\begin{align*}
\|T\|_{S^{p}}\leq \|K\|_{L^p,L^{p^{\prime}}}^{1/2}\|K^{*}\|_{L^p,L^{p^{\prime}}}^{1/2},
\end{align*}
where $p'$ is the conjugate index of $p$ such that $1/p+1/p'=1$, $K^{*}(x,y)=\overline{K(y,x)}$, and $\|\cdot\|_{L^p, L^{p^{\prime}}}$ denotes the mixed-norm:
$
\|K\|_{L^p,L^{p^{\prime}}}:=\big\|\|K(x,y)\|_{L^p(dx)}\big\|_{L^{p^{\prime}}(dy)}.
$
Later, Goffeng (\cite{Goffeng}) proved that the condition $K(x,y)\in  L^{2}(X\times X)$ in the above statement can be removed.

Moreover, Janson--Wolff (\cite[Lemma 1 and Lemma 2]{JW}) extended the above statement to the corresponding weak-type version general measure space $(X,\mu)$: if $p>2$ and $1/p+1/p^{\prime}=1$, then
\begin{align}\label{integral}
\|T\|_{S^{p,\infty}}\leq \|K\|_{L^{p},L^{p^{\prime},\infty}}^{1/2}\|K^{*}\|_{L^{p},L^{p^{\prime},\infty}}^{1/2},
\end{align}
where $\|\cdot\|_{L^p, L^{p^{\prime},\infty}}$ denotes the mixed-norm:
$
\|K\|_{L^p,L^{p^{\prime},\infty}}:=\big\|\|K(x,y)\|_{L^p(dx)}\big\|_{L^{p^{\prime},\infty}(dy)}.
$

Next, back to our Neumann Laplacian setting, we first note that
\begin{align}\label{comb3}
\left\|(b(x)-b(y))K_\ell(x,y)\right\|_{L^p, L^{p^{\prime},\infty}}&\leq \|(b(x)-b(y))K_\ell(x,y)\chi_{\mathbb{R}_{+}^n}(x)\chi_{\mathbb{R}_{+}^{n}}(y)\|_{L^p,L^{p^{\prime},\infty}}\nonumber\\
&\quad+ \|(b(x)-b(y))K_\ell(x,y)\chi_{\mathbb{R}_{-}^n}(x)\chi_{\mathbb{R}_{-}^{n}}(y)\|_{L^p,L^{p^{\prime},\infty}},
\end{align}
where in the last inequality we used the fact that $K_\ell(x,y)=0$ whenever $x$ and $y$ belong to distinct half-plane.

By weak-type Young's inequality, for $1/q=1-2/p$,
\begin{align*}
\|(b(x)-b(y))K_\ell(x,y)\chi_{\mathbb{R}_{\pm}^n}(x)\chi_{\mathbb{R}_{\pm}^{n}}(y)\|_{L^p,L^{p^{\prime},\infty}}&\leq
\left\|\frac{b(x)-b(y)}{|x-y|^{n}}\chi_{\mathbb{R}_{\pm}^n}(x)\chi_{\mathbb{R}_{\pm}^{n}}(y)\right\|_{L^p, L^{p^{\prime},\infty}}\nonumber\\
&\leq \left\|\frac{b(x)-b(y)}{|x-y|^{2n/p}}\chi_{\mathbb{R}_{\pm}^n}(x)\chi_{\mathbb{R}_{\pm}^{n}}(y)\right\|_{L^{p},L^{p}}\left\|\frac{1}{|x-y|^{n(1-2/p)}}\right\|_{L^{\infty},L^{q,\infty}}\nonumber\\
&\leq C\|b_{\pm,e}\|_{B_{p,p}^{n/p}(\mathbb{R}^n)}\\
&\leq C\|b\|_{B_{p,p}^{\frac{n}{p},\Delta_N}(\mathbb{R}^n)}.
\end{align*}
%
%
%Similarly, for the second term,
%\begin{align}\label{verify11}
%\|(b(x)-b(y))K(x,y)\chi_{\mathbb{R}_{-}^n}(x)\chi_{\mathbb{R}_{-}^{n}}(y)\|_{L^p,L^{p^{\prime},\infty}}&\leq
%\left\|\frac{b(x)-b(y)}{|x-y|^{n}}\chi_{\mathbb{R}_{-}^n}(x)\chi_{\mathbb{R}_{-}^{n}}(y)\right\|_{L^p, L^{p^{\prime},\infty}}\nonumber\\
%&\leq \left\|\frac{b(x)-b(y)}{|x-y|^{2n/p}}\chi_{\mathbb{R}_{-}^n}(x)\chi_{\mathbb{R}_{-}^{n}}(y)\right\|_{L^{p},L^{p}}\left\|\frac{1}{|x-y|^{n(1-2/p)}}\right\|_{L^{\infty},L^{q,\infty}}\nonumber\\
%&\leq C\|b_{-}\|_{B_{p,p}^{n/p,e}(\mathbb{R}_{-}^n)}\leq C\|b\|_{B_{p,p}^{\frac{n}{p},\Delta_N}(\mathbb{R}^n)}.
%\end{align}
Therefore,
\begin{align}\label{verify1}
\left\|(b(x)-b(y))K_\ell(x,y)\right\|_{L^p, L^{p^{\prime},\infty}}\leq C\|b\|_{B_{p,p}^{\frac{n}{p},\Delta_N}(\mathbb{R}^n)}.
\end{align}
Similarly,
\begin{align}\label{verify2}\left\|(b(x)-b(y))\overline{K_\ell(y,x)}\right\|_{L^p, L^{p^{\prime},\infty}}\leq C\|b\|_{B_{p,p}^{\frac{n}{p},\Delta_N}(\mathbb{R}^n)}.
\end{align}
Combining the inequalities \eqref{verify1}, \eqref{verify2} and then applying the weak-type Russo's inequality \eqref{integral}, we see that $$\|[b,R_{N,\ell}]\|_{S^{p,\infty}}\leq C\|b\|_{B_{p,p}^{\frac{n}{p},\Delta_N}(\mathbb{R}^n)}.$$ Since this inequality holds for all $n<p<\infty$, we can apply the interpolation $(S^{p_1,\infty},S^{p_2,\infty})_{\theta_p}=S^{p}$ and  $(B_{p_1,p_1}^{\frac{n}{p_1},\Delta_N}(\mathbb{R}^n),B_{p_2,p_2}^{\frac{n}{p_2},\Delta_N}(\mathbb{R}^n))_{\theta_{p}}=B_{p,p}^{\frac{n}{p},\Delta_N}(\mathbb{R}^n)$, where $\frac{1-\theta_p}{p_1}+\frac{\theta_p}{p_2}=\frac{1}{p}$, to obtain that
\begin{align*}
\|[b,R_{N,\ell}]\|_{S^{p}}\leq C\|b\|_{B_{p,p}^{\frac{n}{p},\Delta_N}(\mathbb{R}^{n})}.
\end{align*}
This finishes the proof of sufficient condition for the case $n<p<\infty$.
\end{proof}

\section{Theorem \ref{schatten}: $0<p\leq n$}\label{four}

In this section, we will show the second argument of Theorem \ref{schatten}. That is, for each $\ell\in \{1,2,\ldots,n\}$ and for $0<p\leq n$, we will show that the commutator $[b,R_{N,\ell}]\in S^p$ if and only if $b$ is a constant $c_1$ on $\mathbb{R}^n_+$ and another constant $c_2$ on $\mathbb{R}^n_-$ (in the sense of almost everywhere), where $c_1$ and $c_2$ may not be the same. The key difficulty is to show the necessary part of the Theorem \ref{schatten}. To show this, it suffices to consider the endpoint case $p=n$ since one has the inclusion $S^p \subset S^q$ for $p<q.$

To complete our proof we will use the following lemmas.

\begin{lemma}\label{twocubes}
For any $k\in\mathbb{Z}$ and cube $Q\in \mathcal{D}_{k}^0$ and $a_j=\pm 1(j=1,2,\ldots,2n)$, there are cubes $Q' \in \mathcal{D}_{k+2}^0, Q'' \in \mathcal{D}_{k+2}^0$ such that $Q'\subset Q, Q'' \subset Q$ and if $x=(x_1,x_2,\ldots,x_n)\in Q'$, $y= (y_1,y_2,\ldots,y_n)\in Q''$, then $a_j(x_j-y_j)\geq 2^{-k}$ for $j=1,2,\ldots.,n$.
\end{lemma}
\begin{proof}
Recall that $\mathcal{D}^0$ is a standard
system of dyadic cubes on $\mathbb{R}^n$. Then by symmetry, it suffices to consider the case that $Q$ belongs to the first quadrant. Suppose that the vertex of $Q$ is at $2^{-k}m$ for some $k\in\mathbb{Z}$ and $m=(m_1,\ldots,m_n)\in\mathbb{N}^{n}_+$. Then we pick cubes $Q' \in \mathcal{D}_{k+2,+}^0, Q'' \in \mathcal{D}_{k+2,+}^0$ with vertices at $(2^{-k}m_1+2^{-k-2}+a_12^{-k-2},\ldots,2^{-k}m_n+2^{-k-2}+a_n2^{-k-2})$ and $(2^{-k}m_1+2^{-k-2}-a_12^{-k-2},\ldots,2^{-k}m_n+2^{-k-2}-a_n2^{-k-2})$, respectively. It is direct to verify that these cubes satisfy the properties in the statement.
\end{proof}

Now, we provide a lower bound for a local pseudo oscillation of the symbol $b$ in the commutator.
\begin{lemma}\label{lowerboundcommu}
Let $b\in C^\infty(\mathbb{R}^{n})$. Suppose that there is a point $x_0\in \mathbb{R}^{n}$ such that $\nabla b(x_0) \neq 0$. Then there exist constants $C>0, \epsilon>0$ and  $N>0$ such that if $k>N$, then for any cube $Q \in \mathcal{D}_{k}^0$ satisfying $|{\rm center}(Q)- x_0|<\epsilon$, one has
\begin{equation}\label{contonb}
\left|\fint_{Q'}b -\fint_{Q''}b\right| \geq C2^{-k}|\nabla b(x_0)|,
\end{equation}
where $Q'$ and $Q''$ are the cubes chosen in Lemma \ref{twocubes}.

\end{lemma}

\begin{proof}
We will now denote by $c_Q := (c_Q^1,c_Q^2,\ldots,c_Q^n)$ the center of $Q$ and $x = (x_1, x_2,\ldots,x_n)$, then by Taylor's formula we have
\begin{equation}\label{usetay}
 b(x) = b(c_Q) + \sum_{j=1}^{n}{(\partial_{x_j}b)(c_Q)\over j!}(x_j-c_{Q}^j) + R(x,c_Q),
\end{equation}
where the remainder term $R(x,c_Q)$ satisfies
\begin{equation}\label{laste}
 |R(x,c_Q)|\leq C\sum_{j=1}^n\sum_{k=1}^n \sup_{\theta\in[0,1]}\big|(\partial_{x_j}\partial_{x_k}b)(x+\theta(c_Q-x))\big| |c_Q - x|^{2}.
\end{equation}
Observe that the condition $\theta\in[0,1]$ implies that if $x\in Q$, then
$$|x+\theta(c_Q-x)-c_Q|\lesssim 2^{-k},$$
which implies that for $\epsilon = \epsilon_b >0$ sufficiently small, the right hand side in estimate \eqref{laste} can be absorbed in the right hand side of \eqref{contonb}. Thus, it suffices to deal with the first two terms on the right hand side of \eqref{usetay}.

By  Lemma \ref{twocubes} above, for $x'=(x_1',...,x_n')\in Q'$ and $x''=(x_{1}'',...,x_{n}'')\in Q''$, one has the following
\begin{equation*}
    {\rm sgn}(\partial_{x_j}b)(c_Q)(x_j^{'}-x_j^{''})\geq 2^{-k}, \quad j=1,2,\ldots,n.
\end{equation*}

Hence, we have
\begin{align*}
&\left|\fint_{Q'}b(x')dx' -\fint_{Q''}b(x'')dx''\right|\\ & \geq  \bigg|\fint_{Q'}\fint_{Q''}\sum_{j=1}^{n}{(\partial_{x_j}b)(c_Q)\over j!}(x_j^{'}-x_j^{''})dx''dx'\bigg| - \fint_{Q'}|R(x',c_Q)|dx' - \fint_{Q''}|R(x'',c_Q)|dx''  \\&\geq C\sum_{j=1}^{n}(\partial_{x_j}b)(c_Q)2^{-k} -C2^{-2k}\|\nabla b\|_{L^\infty(B(x_0,1))}\\
&\geq C2^{-k} |\nabla b(x_0)|.
\end{align*}
This completes the proof of Lemma \ref{lowerboundcommu}.
\end{proof}
For any $h\in B(0,1)$, define the conditional expectation of a locally integrable function $f$ on $\mathbb{R}^{n}$  with respect to the increasing family of $\sigma-$algebras $\sigma(\mathcal{D}_{k}^0)$ and $\sigma(\tau^h\mathcal{D}_{k}^0)$    by the expression:
$$E_{k}^0(f)(x)=\sum_{Q\in \mathcal{D}_{k}^0}(f)_{Q}\chi_{Q}(x),\ x\in\mathbb{R}^n,$$
$$E_{k,h}^0(f)(x)=\sum_{Q\in \tau^h\mathcal{D}_{k}^0}(f)_{Q}\chi_{Q}(x),\ x\in\mathbb{R}^n,$$
respectively. Then we have the following Lemma.

\begin{lemma}\label{aux}
A function $b \in \mathcal{M}(\mathbb{R}^n)$ is a constant on $\mathbb{R}_{\pm}^n$ if there exist constants $C>0$ and $\ell\in\{1,2,...,n\}$ such that
\begin{align}\label{eqsup1}
\sup_{h\in B(0,1)}\bigg\|\bigg\{\fint_{Q}\fint_{Q}|E_{k+2,h}^0(b)(x') - E_{k+2,h}^0(b)(x'')|dx'dx''\bigg\}_{\substack{Q\in \tau^h\mathcal{D}_\pm^0\\Q\subseteq \mathbb{R}_\pm^n}}\bigg\|_{l^{n}} \leq C\|[b,R_{N,\ell}]\|_{S^n}.
\end{align}
(In the display $Q\in \tau^h\mathcal{D}_{k,\pm}^0$ and $Q\subseteq \mathbb{R}_\pm^n$, and both $Q$ and $k$ vary and $\tau^h$ denotes translation by $h$.)
\end{lemma}

\begin{proof}
Pick a smooth compactly supported function $\psi$ over $\mathbb{R}^n$ which integrates to 1 and pick $\varepsilon$ be a small positive constant. Consider $\psi_{\varepsilon}(x) = \frac{1}{\varepsilon^n}\psi(\frac{x}{\varepsilon^n})$ and $b_{\varepsilon} := b*\psi_{\varepsilon}$, where $0<\varepsilon<1$. Then $b_{\varepsilon}$ is a smooth function and converges to $b$ almost everywhere.

%The assumption is that $b \in L_{\rm loc}^1(\mathbb{R}^{n})$, but previously we required that $b$ is smooth. We will denote by $\psi_{\epsilon}(x) = \epsilon^{-n}\psi_{\epsilon}(\delta_{\epsilon^{-1}}x)$, where $\psi$ is a smooth compactly supported bump function which integrates to 1 and the $\epsilon$ is small positive constant. Then consider $b_{\epsilon} = b*\psi_{\epsilon}$, which is smooth and we will argue that these are all constant and they converge to $b$ pointwise, so this is sufficient.
%By  \eqref{eqsup1},
%\begin{equation}\label{infi}
%\bigg\|\bigg\{\fint_{Q}\fint_{Q}|E_{k+2}^0(b_{\varepsilon})(x') - E_{k+2}^0(b_{\varepsilon})(x'')|dx'dx''\bigg\}_{Q\in \mathcal{D}_\pm^0}\bigg\|_{l^{n}}   \leq C\|[b,R_{N,\ell}]\|_{S^n}.
%\end{equation}

Now we claim that $b_\varepsilon$ is a constant on $\mathbb{R}_{\pm,\varepsilon}^n$, where
$$\mathbb{R}_{\pm,\varepsilon}^n:=\mathbb{R}_\pm^n\pm\varepsilon e_n,\ \ {\rm for}\ e_n=(0,0,...,1).$$
%Now we show that if $b_{\epsilon}$ is not constant, then the norm above is infinite,  which gives us a contradiction.
If not, then observe that there exists a point $x_0=(x_0^{(1)},\ldots,x_0^{(n)})\in \mathbb{R}_{\pm,\varepsilon}^{n}$ such that $\nabla b_{\varepsilon}(x_0) \neq 0$. By Lemma \ref{lowerboundcommu}, there exist some $\epsilon>0$ and $N>0$ such that if $k>N$, then for any cube $Q\in \mathcal{D}_{k,\pm}^0$ satisfying $|{\rm center}(Q)-x_0|<\epsilon$,
\begin{equation*}
\fint_{Q}\fint_{Q}|E_{k+2}^0(b_{\varepsilon})(x') - E_{k+2}^0(b_{\varepsilon})(x'')|dx'dx'' \gtrsim 2^{-k}|\nabla b_\varepsilon(x_0)|.
\end{equation*}
Denote $\mathcal{A}_k^+(x_0)$ be the set consisting of $Q\in \mathcal{D}_{k,+}^0$ satisfying $|{\rm center}(Q)-x_0|<\epsilon$ and $d(Q,\partial \mathbb{R}_+^n)\geq x_0^{(n)}$. Similarly, denote $\mathcal{A}_k^-(x_0)$ be the set consisting of $Q\in \mathcal{D}_{k,-}^0$ satisfying $|{\rm center}(Q)-x_0|<\epsilon$ and $d(Q,\partial \mathbb{R}_-^n)\leq x_0^{(n)}$. Then observe that for any $k>N$, the number of $\mathcal{A}_k^\pm(x_0)$ is at least $2^{kn}$, which implies that
\begin{align*}
\sum_{k>N}\sum_{Q\in \mathcal{A}_k^\pm (x_0)}\left(\fint_{Q}\fint_{Q}|E_{k+2}^0(b_{\varepsilon})(x') - E_{k+2}^0(b_{\varepsilon})(x'')|dx'dx'' \right)^n=+\infty.
\end{align*}
However, the left hand side above is dominated by
\begin{align}\label{gjkl}
&\sup\limits_{h\in B(0,\varepsilon)}\sum_{k\in\mathbb{Z}}\sum_{Q\in \mathcal{A}_k^\pm (x_0)}\left(\fint_{Q}\fint_{Q}|E_{k+2}^0(\tau^hb)(x') - E_{k+2}^0(\tau^hb)(x'')|dx'dx'' \right)^n\nonumber\\
&=\sup\limits_{h\in B(0,\varepsilon)}\sum_{k\in\mathbb{Z}}\sum_{Q\in \mathcal{A}_k^\pm (x_0)}\left(\fint_{Q}\fint_{Q}|\tau^hE_{k+2,h}^0(b)(x') - \tau^hE_{k+2,h}^0(b)(x'')|dx'dx'' \right)^n\nonumber\\
&=\sup\limits_{h\in B(0,\varepsilon)}\sum_{k\in\mathbb{Z}}\sum_{Q\in \mathcal{A}_k^\pm (x_0)}\left(\fint_{\tau^hQ}\fint_{\tau^hQ}|E_{k+2,h}^0(b)(x') - E_{k+2,h}^0(b)(x'')|dx'dx'' \right)^n.
\end{align}
Note that the restriction $x_0\in\mathbb{R}_{\pm,\varepsilon}^{n}$ implies that for any $h\in B(0,\varepsilon)$ and $Q\in \mathcal{A}_k^\pm (x_0)$, one has $\tau^hQ\subseteq \mathbb{R}^n_\pm$. This, together with inequality \eqref{eqsup1}, implies that the right hand side of inequality \eqref{gjkl} is dominated by $\|[b,R_{N,\ell}]\|_{S^n}^n$, which is a contradiction.
This ends the proof of Lemma \ref{aux}.
\end{proof}
\begin{proposition}\label{kkkey}
Suppose $b\in \mathcal{M}(\mathbb{R}^n)$. Then for any $\ell\in \{1,2,\ldots,n\}$, the commutator $[b,R_{N,\ell}]\in S^n$ if and only if $b\equiv {\rm Const_1}$ on $\mathbb{R}^n_+$ a.e. and $b\equiv {\rm Const_2}$ on $\mathbb{R}^n_-$ a.e..
\end{proposition}

\begin{proof}
Note that if $b\equiv {\rm Const_1}$ on $\mathbb{R}^n_+$ and $b\equiv {\rm Const_2}$ on $\mathbb{R}^n_-$ a.e, then
$[b,R_{N,\ell}]=0$ since $K_\ell(x,y)=0$ whenever $x$ and $y$ belong to distinct half-plane. Hence it remains to  consider the direction in which we assume $[b,R_{N,\ell}]\in S^n$. To this end, by Lemma \ref{step1},
there exists a constant $C>0$ such that
\begin{align}\label{condineq}
&\sup_{h\in B(0,1)}\bigg\|\bigg\{\fint_{Q}\fint_{Q}|E_{k+2,h}^0(b)(x') - E_{k+2,h}^0(b)(x'')|dx'dx''\bigg\}_{\substack{Q\in \tau^h\mathcal{D}_\pm^0\\Q\subseteq \mathbb{R}_\pm^n}}\bigg\|_{l^{n}}\nonumber\\
&\leq C\sup_{h\in B(0,1)}\bigg\|\bigg\{\fint_{Q}|E_{k+2,h}^0(b)(x') - E_{k,h}^0(b)(x')|dx'\bigg\}_{\substack{Q\in \tau^h\mathcal{D}_\pm^0\\Q\subseteq \mathbb{R}_\pm^n}}\bigg\|_{l^{n}}\nonumber\\
&\leq C\|[b,R_{N,\ell}]\|_{S^n}.
\end{align}
(In the display $Q\in \tau^h\mathcal{D}_{k,\pm}^0$, $Q\subseteq \mathbb{R}_\pm^n$, and both $Q$ and $k$ vary.)
%Observe that this is a corollary of Lemma ??? and can be seen as a general remark. Consider a random variable $X$, we have for $1\leq p <\infty$,
%\begin{equation*}
%    \|X-\mathbb{E}X\|_{p} \simeq \|X-X'\|_{p},
%\end{equation*}
%where $X'$ is an independent copy of $X$. Also,
%\begin{align*}
% &\|X-\mathbb{E}X\|_{p} \simeq \|X-\mathbb{E}X'\|_{p}\\
% &\leq \|X-X'\|_{p} \leq 2 \|X-\mathbb{E}X\|_{p}.
%\end{align*}
%Here the first inequality is due to convexity and second is by the triangle inequality.
%
%Thus, Lemma ??? implies
This, together with Lemma \ref{aux}, finishes the proof of Proposition \ref{kkkey}.
\end{proof}

\bigskip

 \noindent
 {\bf Acknowledgements:}

M. Lacey is a 2020 Simons Fellow, his Research is supported in part by grant  from the US National Science Foundation, DMS-1949206. J. Li is supported by the Australian Research Council through the research grant DP220100285.  B. D. Wick's research is supported in part by U. S. National Science Foundation -- DMS 1800057, 2054863, and 2000510 and Australian Research Council -- DP 220100285.

\bibliographystyle{plain}
\bibliography{references}

%
%\begin{thebibliography}{99999}
%\bibitem{DDSY} D. Deng, X. Duong, A. Sikora and L. Yan, Comparison of the classical BMO with the BMO spaces associated
%   with operators and applications,  {\it Rev. Mat. Iberoam}. {\bf 24} (2008),  267--296.
%
%\bibitem{HKZ} B. Hu, L.H. K and K. Zhu, Frames and operators in Schatten classes, {\it Houston J. Math}. {\bf 41} (2015), 1191-1219.
%
%\bibitem{Journe} J.L. Journ\'{e}, Calder\'{o}n-Zygmund operators, pseudodifferential operators and the Cauchy integral of Calder\'{o}n, {\it lecture notes in mathematics}. {\bf 994}. Springer-Verlag, Berlin, (1983).
%
%\bibitem{LW} J. Li and B.D. Wick, Characterizations of $H_{\Delta_{N}}^{1}(\mathbb{R}^{n})$ and $BMO_{\Delta_{N}}(\mathbb{R}^{n})$ via weak factorizations and commutators. {\it J. Funct. Anal}. {\bf 272} (2017), 5384--5416.
%
%    \bibitem{RS}  R. Rochberg  and S. Semmes,  Nearly weakly orthonormal sequences, singular value estimates, and Calderon-Zygmund operators, {\it J. Funct. Anal}., {\bf86} (1989), 237--306.
%
%\bibitem{Simon} B. Simon, Trace ideals and their applications. {\it Cambridge Univ. Press}. Cambridge, (1979).
%
%\bibitem{S} W. A. Strauss, Partial differential equation: An introduction, {\it John Wiley \& Sons, Inc., New York}. (2008), xiv+454.
%\end{thebibliography}

\end{document}